\documentclass[11pt,fleqn]{article}

\usepackage[utf8]{inputenc}
\usepackage[T1]{fontenc}

\usepackage{geometry}
\usepackage[all]{xy}
\usepackage{amsmath,amsthm,amsfonts,amssymb,stmaryrd}
\usepackage{latexsym}
\usepackage{graphicx}
\usepackage{float} 
\usepackage{xcolor}
\usepackage{enumitem}
\usepackage[small, bf, margin=90pt, tableposition=bottom]{caption}
\usepackage{mathpazo}
\usepackage{euler}
\usepackage{mathdots}
\usepackage{mathtools}

\usepackage[
  colorlinks,
  breaklinks,
  naturalnames,
  ]{hyperref}

\usepackage{latexsym}
\usepackage{algorithm}

\newtheorem{thm}{Theorem}[section]

\newtheorem{prop}[thm]{Proposition}
\newtheorem{cor}[thm]{Corollary}
\newtheorem{rem}[thm]{Remark}

\newtheorem{defn}[thm]{Definition}
\newtheorem{exmp}[thm]{Example}

\makeatletter\label{e:dispaa}
\label{e:dispau}
\label{e:dispav}
\label{e:dispaw}
\label{e:dispax}
\def\theequation{\thesection.\@arabic\c@equation}
\makeatother\makeatletter

\def\cal#1{\mathcal{#1}}

\newcommand{\ca}{\cal{A}}
\newcommand{\cb}{\cal{B}}

\newcommand{\ci}{\cal{I}}
\newcommand{\cm}{\cal{M}}

\newcommand{\mcst}{(\mC_0 \times \mC_0)^*}

\newcommand{\benu}{\begin{enumerate}}
\newcommand{\enu}{\end{enumerate}}

\newcommand{\beqna}{\begin{eqnarray}}
\newcommand{\eqna}{\end{eqnarray}}
\newcommand{\beqnast}{\begin{eqnarray*}}
\newcommand{\eqnast}{\end{eqnarray*}}
\newcommand{\beqn}{\begin{equation}}
\newcommand{\eqn}{\end{equation}}
\newcommand{\beqnst}{\begin{equation*}}
\newcommand{\eqnst}{\end{equation*}}

\makeatother

\newcommand{\mA}{{\mathcal A}}
\newcommand{\mB}{{\mathcal B}}
\newcommand{\mC}{{\mathcal C}}
\newcommand{\mD}{{\mathcal D}}

\newcommand{\bema}{\left ( \begin{array}}
\newcommand{\ema}{\end{array} \right )}

\newcommand{\Hom}{\operatorname{Hom}}
\newcommand{\End}{\operatorname{End}}

\newcommand{\cat}{\operatorname{cat}}

\newcommand{\xAx}{\ _{x}{\cal A}_{x}}
\newcommand{\yAy}{\ _{y}{\cal A}_{y}}
\newcommand{\yAx}{\ _{y}{\cal A}_{x}}
\newcommand{\zAy}{\ _{z}{\cal A}_{y}}

\newcommand{\yBx}{\ _{y}{\cal B}_{x}}
\newcommand{\zBy}{\ _{z}{\cal B}_{y}}
\newcommand{\zBx}{\ _{z}{\cal B}_{x}}

\newcommand{\xCx}{\ _{x}{\cal C}_{x}}
\newcommand{\yCy}{\ _{y}{\cal C}_{y}}
\newcommand{\yCx}{\ _{y}{\cal C}_{x}}
\newcommand{\zCy}{\ _{z}{\cal C}_{y}}

\newcommand{\xDx}{\ _{x}{\cal D}_{x}}
\newcommand{\yDx}{\ _{y}{\cal D}_{x}}

\newcommand{\morita}{\cal{D}(P,{\cal C},Q)}

\newcommand{\yIx}{\ _{y}{\cal I}_{x}}
\newcommand{\zIy}{\ _{z}{\cal I}_{y}}

\newcommand{\xLx}{\ _{x}\Lambda_{x}}

\newcommand{\yfx}{\ _{y}f_{x}}
\newcommand{\zfy}{\ _{z}f_{y}}
\newcommand{\ygx}{\ _{y}g_{x}}
\newcommand{\zgy}{\ _{z}g_{y}}
\newcommand{\xex}{\ _{x}e_{x}}
\newcommand{\yey}{\ _{y}e_{y}}
\newcommand{\zez}{\ _{z}e_{z}}
\newcommand{\xegx}{\ _x^{}e_x^g}
\newcommand{\yegy}{\ _y^{}e_y^g}

\newcommand{\ug}{g}

\newcommand{\xidx}{\ _{x}1_{x}}
\newcommand{\yidy}{\ _{y}1_{y}}
\newcommand{\zidz}{\ _{z}1_{z}}
\newcommand{\back}{\hspace{-3pt}}

\newcommand{\otu}{\underline{\otimes}}
\newcommand{\ot}{\otimes}

\newcommand{\ASH}{\underline{{\cal A} \# H}}
\newcommand{\xASHx}{\ _x(\ASH) _x}
\newcommand{\yASHx}{\ _y(\ASH) _x}

\newcommand{\vect}{k\rm{-Vect}}
\newcommand{\alg}{\operatorname{Alg}}

\newcommand{\btil}{g}
\newcommand{\atil}{f}

\begin{document}

\title{Partial Hopf module categories}

\author{Edson R.~Alvares, Marcelo M.S.~Alves, Eliezer Batista 
\thanks{\footnotesize This work has been supported by the projects Projeto Mobilidade AUGM (Grupo Montevideo), Prosul CNPq no. 490065/2010-4, Funda\c cao Araucaria 490/16032, CNPq 304705/2010-1.}}

\maketitle
\date{}

\begin{abstract} The effectiveness of the aplication of constructions in 
$G$-graded $k$-categories to the computation of the fundamental group of a finite dimensional $k$-algebra, 
alongside with open problems still left untouched by those methods 
and new problems arisen from the introduction of the concept of fundamental group of a $k$-linear category,  
motivated the investigation of 
 $H$-module categories, i.e.,  actions of a Hopf algebra $H$ on a $k$-linear category. The $G$-graded case corresponds then to actions of the Hopf algebra $k^G$
on a $k$-linear category.  
In this work we take a step further and introduce partial $H$-module categories. We extend several results of partial $H$-module algebras to this context, 
such as the globalization theorem, the construction of the partial smash product and 
the Morita equivalence of this category with the smash product over a globalization. 
We also present a detailed description of partial actions of $k^G$.

 \end{abstract}


\maketitle
\date{}

\section{Introduction}

Group graded rings and algebras became a very important topic of research and several fundamental results were proved in this subject, for example, the Cohen Montgomery duality theorem \cite{CoM}. In fact, the concept of graded ring can be extended to a categorical level, giving rise to group graded $k$ linear categories and  there are several  problems leading to this concept, for example the problem of finding Galois coverings of a category defined by a quiver. Let $A$ be a basic, connected, finite dimensional algebra over an algebraically closed field $k$. By a result of Gabriel\cite{G}, there exists a unique finite connected quiver
$Q$ and a two-sided ideal $I$ of the path algebra $kQ$, such that $A \cong kQ/I$. Such a pair $(Q,I)$ is called a bound quiver. 
 The morphism $v:kQ\rightarrow A\cong kQ/I$  as well as  $(Q, I)$ are both called presentations of $A$. Following \cite{MVP}, one can define the
fundamental group $\pi_{1} (Q, I)$. This quiver can also be considered as a $k$-category, denoted by $\mC$ then, the fundamental group $\pi_{1} (Q, I)$  can be viewed as the fundamental group of $\mC$.  In this situation we may construct a covering $\mC\#G$ as it is done in \cite{CM}, that we shall call a {\em smash covering}. Given $\mC$, the full subcategory of the category of Galois coverings with a fixed point whose objects are the smash coverings is equivalent to the whole category, so the computation of the fundamental group of $\mC$ may be restricted to this subcategory \cite{CRS}. This fundamental group coincides with the one previously defined by Bongartz and Gabriel \cite{BG} in certain situations, for example when a universal covering exists.
A universal covering corresponds to a universal grading, i. e. a grading  such that any other one is a quotient of this grading (there are several algebras, such as the matrix algebras, for which such a universal grading does not exist  - see \cite{CRS} for details).

The most apropriated language to treat $G$ graded rings and algebras is of Hopf algebras. A $G$ graded algebra $A$ can be viewed as a $k^G$ module algebra, then it is natural to consider $G$ graded categories as appropriated $k^G$ module categories, more generally, we can define, for any Hopf algebra $H$ the concept of $H$ module category,  where by category we mean $k$-linear category, this is done in \cite{CS}.

Another rich source of interesting problems is to consider partial actions. Partial actions of groups were first proposed  by R.~Exel and other authors in the context of $C^*$-algebras, see for example \cite{E}, and afterwards they appeared in a purely algebraic setting (\cite{DE, E2, FL, Lo}). Partial actions of groups became important as a tool to characterize algebras as partial crossed products. Soon after, the theory of partial actions was extended to the Hopf algebraic setting \cite{CJ}. The most basic example of partial Hopf action is given by a partial action of a group $G$ on a unital algebra $A$  in which the partial domains are unital ideals in $A$, this gives a partial action of the group algebra $kG$ on $A$.  One of the main concerns when working with partial Hopf actions is to exhibit nontrivial examples
of them which are neither usual actions nor coming from partial group actions. For the case of $H=U(\mathfrak g)$, the enveloping algebra of a Lie algebra, it is known \cite{CJ} that every partial action is in fact a global action. The second and the third authors gave in \cite{AB1} an example of a partial action of the dual group algebra $k^G$ on an ideal of $kG$, this can be considered as a partial $G$ grading.

The idea underlying this article is to consider a more general setting, that is: partial $G$ gradings of $k$-categories as a special case of partial Hopf-module categories, in particular, when the category has only a single object, this characterizes partial $G$ gradings on algebras. Given a Hopf algebra $H$, the concept of a partial $H$ module category unifies both  Hopf module categories, when the partial action is global, and  partial $H$ module algebras, when the category has a single object. We develop in this article a procedure that may be used to obtain nontrivial examples of partial $H$ module categories in a systematic way. We apply it to different situations, for instance, the partial actions of $k^G$ over the base field $k$ are fully classified and then lifted to partial $k^G$ actions on an arbitrary $k$ category, making them $k^G$ module categories. We show also that the examples of partial action with $H=H_4$, the Sweedler $4$-dimensional Hopf algebra, cover all possible ones.

The contents of the article are as follows.

In Section \ref{toolkit}
we deal with some prerequisites.

In Section \ref{lambdas y ejemplos}
the definition of partial $H$-module category is given, and afterwards we reformulate it so as to handle the conditions in an easier way. This reformulation leads to a detailed description of partial actions for $H=k^G$  (see Theorem \ref{graduacoesparciais}). With the same techniques,  it is also possible to classify all partial actions of the Sweedler Hopf algebra on the base field $k$. In the subsequent subsections we discuss different kinds of examples, including the partial group actions on $k$ categories \cite{FL} and  induced partial Hopf actions.

In Section \ref{globalization} we prove (see Theorem \ref{minimal globalization}) that every partial $H$-module category has a minimal 
globalization, unique up to isomorphism.
It is natural, given a partial $H$-module category $\mC$ to define a partial smash product
$\underline{\mC\#H}$ and this is what we do in Section \ref{partial smash}, where we also study its relation with the smash product $\mD\#H$, where $\mD$ is a globalization of $\mC$, proving in Section \ref{Morita} that there is a Morita equivalence. The proof of Morita equivalence between the partial smash product and the globalized one, unlike in the algebraic case, done in \cite{AB1}, required categorical tools, more specifically a result for characterizing Morita equivalence between $k$-categories proved in \cite{CS}.

\section{A categorical toolkit}\label{toolkit}
 
This section, as its title suggests, contains definitions and some basic properties that will be used throughout the article, as well as some examples.
We refer the reader to \cite{BG} for the definition of $k$-categories. In what follows, if $\mA$ is a category then $\mA_0$ and $\mA_1$ denote  respectively the objects and morphisms of $\mA$. 

\begin{defn} Let $X$ be a non-empty set. A $k$-category over $X$, or, shortly, an $X$-category is a $k$-category $\mA$ with $\mA_0=X$. Given $X$-categories $\mA$ and $\mB$ and a $k$-functor $F: \mA \rightarrow \mB$, 
we will say that $F$ is an $X$-functor if $F(x) = x$ for each $x \in X = \mA_0 = \mB_0$. 
\end{defn}

Note that $X$-categories with $X$-functors as morphisms form a category, let us denote it as $X-\cat$. Most of the constructions in this article involve this category.

\begin{exmp}\label{example_X_finite} If $X$ is a finite set of $n$ elements, say $X = \{1, \ldots, n\}$,  and $\mA$ is an $X$-category, then we may construct as in \cite{CS} the matrix algebra 
\[
a(\mA )=\left\{ \left. ({}_yf_x )_{y,x} \, \right| \, {}_yf_x \in {}_y\mA_x \right\} 
\]
where the product is given by  $({}_yf_x) ({}_yg_x) = (\sum_z {}_yf_z \circ {}_zg_x) $. Note also that 
if $\xidx^\mA$ is the identity morphism of $x \in X$ then the matrix  $e_x = \xidx^\mA \ E_{x,x}$ is an idempotent 
of $a(\mA)$, and the collection $\{e_x\}_{x \in X}$ is a complete set of (nontrivial) orthonormal idempotents  of $a(\mA)$. Moreover, if $\mB$ is another $X$-category, every functor  $F: \mA \rightarrow \mB$ induces canonically the algebra morphism $(\yfx) \mapsto (F(\yfx))$, which takes the idempotent $e_x = \xidx^\mA \ E_{x,x} \in a(\mA)$ to $f_x = \xidx^\mB \ E_{x,x} \in a(\mB)$.  

This construction suggests the following $k$-category: let $X-\alg$ denote the category whose objects are pairs $(A,\{e_x\}_{x \in X})$, where $A$ is a $k$-algebra and  
$\{e_x\}_{x \in X}$ is a  complete system of $n$ orthogonal idempotents of $A$, and a morphism $\varphi : (A,\{e_x\}_{x \in X}) \rightarrow (B,\{f_x\}_{x \in X})$ in
$X-\alg$  is an algebra morphism $\varphi : A \rightarrow B$ such that $\varphi(e_x) = f_x$ for each $x \in X$. The remarks above 
show that the definition of $a(\mA)$ induces an category isomorphism  $\Psi : X-\cat \rightarrow X-\alg$ .

In particular, when $X$ is a unitary set  the category $X-\cat$ coincides with the category of unital algebras with unital morphisms.
\end{exmp}

A non-zero \textbf{walk} $w(f_n,\epsilon_n), \dots ,(f_1,\epsilon_1)$ in $\mC$ is a sequence of linked non-zero virtual morphisms, that is: each $f_i$ is a morphism in the category and each $\epsilon_i$ is $\pm 1$, $s(f_{i}, 1)= s(f_i)$, $t(f_{i}, 1)= t(f_i)$, $s(f_{i}, -1)= t(f_i)$ and $t(f_{i}, -1)= s(f_i)$. The target object of each virtual morphism of $w$ coincides with the source object of the following one, namely $s(f_{i+1}, \epsilon_{i+1})= t(f_{i}, \epsilon_{i})$ for $i=1,\dots ,n-1$. The walk $w$ has source object $s(f_1, \epsilon_1)$ and target object $t(f_n, \epsilon_n)$ and we say that these objects are joined by $w$. 
 
\begin{defn} A $k$-category $\mC$ is called {\bf connected} if given any pair of objects $x$ and $y$, there is a non-zero walk in $\mC$ from $x$ to $y$. If $\mC$ is not connected, then $\mC_{0}$ is the disjoint union of subsets $X_i$, and $\yCx =0$ whenever $x,y$ are not both in the same component $X_i$. 
\end{defn}

\begin{defn} \cite{M}
A {\bf $k$-semicategory} or {\em $k$-category not necesssarily with units} $\mC$ consists of
a collection $\mC_0$ of objects,
a collection $\mC_1$ of morphisms,
two maps $s,t:\mC_1 \to \mC_0$ called {\em source} and {\em target},
such that composition of morphisms, whenever possible, is associative.
We also suppose that for each $x$ and $y \in \mC_0$, the morphism space $\yCx:=\Hom_{\mC}(x,y)$ is a $k$-vector space.
\end{defn}

There are lots of natural examples of semicategories.  We will exhibit our motivating example in Subsection \ref{ejemploideal}, 
but first we mention two examples  related to non-unital algebras.

\begin{exmp}
Let $A$ be a non unital $k$-algebra , then, taking $\mC_0=\{\bullet \}$ and $\mC_1=A$, we obtain a $k$-semicategory.
\end{exmp}

\begin{exmp} Let $A$ and $B$ be two nonunital idempotent $k$-algebras, that is $A^2 =A$ and $B^2 =B$. It is possible to define a Morita context between $A$ and $B$ in this case \cite{GS}. A Morita context is a sextuple $(A,B,M,N, \tau , \sigma)$, where $M$ is a unital $A-B$ bimodule (that is, it verifies $M=AM=MB$),  $N$ is a unital $B-A$ bimodule, the $k$-linear map $\tau : M\otimes_{B} N\rightarrow A$ is an $A-A$ bimodule map and the $k$-linear map $\sigma : N\otimes_{A}M \rightarrow B$ is a $B-B$ bimodule map such that
\begin{eqnarray}
\label{associativityofmorita}
m_1 \sigma (n\otimes m_2) &=& \tau (m_1 \otimes n) m_2, \qquad \forall m_1 , m_2 \in M, \; n\in N,\nonumber \\
n_1 \tau (m\otimes n_2 )  &=& \sigma (n_1 \otimes m ) n_2, \qquad \forall m\in M, \; n_1 , n_2 \in N.
\end{eqnarray}
The above data provides a new associative algebra, the {\em linking algebra}
\[
\mathcal{A}=\left( \begin{array}{cc} A & N \\ M & B \end{array} \right) ,
\]
with multiplication given by
\begin{equation}
\label{linking}
\left( \begin{array}{cc} a_1 & n_1 \\ m_1 & b_1 \end{array} \right)
\left( \begin{array}{cc} a_2 & n_2 \\ m_2 & b_2 \end{array} \right) =
\left( \begin{array}{cc} a_1 a_2 +\sigma (n_1 \otimes m_2 ) & a_1 n_2 + n_1 b_2 \\ 
m_1 a_2 +b_1 m_2 & b_1 b_2 +\tau (m_1 \otimes n_2 )\end{array} \right) .
\end{equation}
The associativity of $\mathcal{A}$ is easily verified using relations 
(\ref{associativityofmorita}). 

The algebra $\mathcal{A}$ can be viewed as a $k$-semicategory with $\mathcal{A}_0 =\{ A,B \}$, $\mathcal{A}_1$ given by the morphisms spaces ${}_{A}\mathcal{A}_{A}=A$, ${}_{B} \mathcal{A}_{B}=B$, ${}_{A}\mathcal{A}_{B}=M$ and ${}_{B}\mathcal{A}_{A}=N$ and compositions compatible with matrix multiplication 
(\ref{linking}).
\end{exmp}

\medskip

Similarly , we define:

\begin{defn}
A {\bf $k$-semifunctor} $F$ from a $k$-semicategory $\mC$ to a $k$-semicategory $\mD$ is a $k$-linear map sending each object $x\in \mC_0$ 
to an object $F(x)\in \mD_{0}$ and each morphism $f:x\to y$ in $\yCx$ to a morphism $F(f)$ in $_{F(x)}\mD_{F(y)}$, such that $F$ preserves composition.
\end{defn}

One of the main  $k$-semifunctors that we are going to use is: 

\begin{exmp}
Take $F$ as the embedding of a $k$-linear category into another $k$-linear category, where $F$ maps the identity at some object $x\in \mC_0$ to a nontrivial idempotent endomorphism of $F(x)$. Then $F$ is a semifunctor but not a functor.
\end{exmp}

\begin{exmp}
A slightly different situation is the following: consider a $k$-algebra $A$ with a 
central idempotent $e$ and define de ideal $I=eAe$.
Let $\mC_{A}$ be the $k$-category with one object and $A$ as $k$-vector space 
of morphisms, and let $\mC_{I}$  be the $k$-semicategory with one object and $I$ as 
set of morphisms.
Then the embedding of $\mC_{I}$ into $\mC_{A}$ is a $k$-semifunctor.
\end{exmp}

We will also need to consider equivalences of semicategories.

\begin{defn} Given two $k$-semifunctors $F,G :\mathcal{A}\rightarrow \mathcal{B}$, a {\bf natural transformation} between $F$ and $G$ is a family of morphisms $\{ \alpha_x \in {}_{G(x)}\mathcal{B}_{F(x)} |x\in \mathcal{A}_0 \}$ such that, for any 
${}_y f_{x} \in {}_y \mathcal{A}_{x}$ the following diagram commutes
\[
\xymatrix{
F(x) \ar[d]_{\alpha_x } \ar[r]^{F({}_{y}f_{x})}& F(y)  \ar[d]^{ \alpha_y }\\
G(x) \ar[r]^{G({}_{y}f_{x})}& G(y) \\
} 
\]
A natural transformation, $\alpha$, between two semifunctors $F$ and $G$ is a {\bf natural isomorphism} if $\alpha_x \in {}_{G(x)}\mathcal{B}_{F(x)}$ is an isomorphism for every $x\in \mathcal{A}_0$. A {\bf semiequivalence} between two $k$-semicategories $\mA$ and $\mB$ is a $k$-semifunctor $F:\mA \to \mB$ such that there exists   
a $k$-semifunctor $G:\mB \to \mA$ verifying that $F\circ G$ and $G\circ F$ are naturally isomorphic to the respective identity semifunctors. 
\end{defn}
We are also going to use the notions of $X$-semicategory and $X$-semifunctor.
\smallskip

Finally, we recall the concept of ideal of a $k$-category.

\begin{defn} An {\bf ideal} ${\cal I}$ of a $k$-category $\mC$ is a
collection of subvector spaces ${}_y{\cal I}_x$ of each morphism space ${}_y\mC_x$, such that the image of the composition map ${}_z\mC_y \otimes {}_y{\cal I}_x \to {}_z\mC_x$ is contained in ${}_z{\cal I}_x$ and the image of the composition map 
${}_y{\cal I}_x \otimes {}_x\mC_u \to {}_y\mC_u$ is contained in ${}_y{\cal I}_u$ for each choice of objects. 
\end{defn}

The same notion makes sense in $X$-semicategories.

\section{Partial $H$-module categories}\label{lambdas y ejemplos}

In this section we are going to define our main object of study, namely, a partial action of a Hopf algebra $H$ on a $k$-category $\mA$. This definition is motivated both by the notion of a partial action of a Hopf algebra $H$ on a unital algebra $A$, as first introduced in \cite{CJ}, and by the concept of $H$-module category, as defined in \cite{CS}, which we recall next. 

A $k$-category $\mA$ is an $H$-{\bf module category} if for each $x, y \in {\mA}_{0}$ there is a $k$-map 
$H \otimes {}_{y}{\mA}_{x} \rightarrow {}_{y}{\mA}_{x} $
sending $h \otimes f$ to $h \cdot f$,  such that 
 
\begin{align}
  \label{H1:condicion1}
& 1_{H} \cdot \yfx = \yfx, \\
  \label{H2:condicion2}
& h \cdot ( \zfy \circ \ygx) =  \sum (h_{(1)} \cdot\zfy) \circ (h_{(2)} \cdot \ygx), \\
  \label{H3:condicion3}
& h \cdot (k \cdot \yfx)  =  (hk) \cdot \yfx, \\
\label{H4:condicion4}
& h \cdot \ _{x}1_{x}  =  \epsilon(h) \cdot \ _{x}1_{x}.
\end{align}

\subsection{Partial $H$-module categories}

\begin{defn}\label{partialHmodulecat}
Let $\mA$ be a $k$-linear category and $H$ a Hopf algebra. We say that there exists a {\bf partial action} of $H$ on $\mA$ if:
\benu
\item $H$ acts trivially on objects. 
\item For every $x, y \in \mA_0$, there is a $k$-linear map $\alpha: H \otimes \yAx \rightarrow \yAx$, denoted by 
$\alpha(h \otimes f) = h \cdot f$, satisfying:
  \begin{align}
  \label{partialH1:condicion1}
& 1_{H} \cdot \yfx = \yfx, \\
  \label{partialH2:condicion2}
& h \cdot ( \zfy \circ \ygx) = \sum (h_{(1)} \cdot\zfy) \circ (h_{(2)} \cdot \ygx), \\
  \label{partialH3:condicion3}
& h \cdot (k \cdot \yfx) = \sum (h_{(1)} \cdot  \ \yidy) \circ ((h_{(2)}k) \cdot  \yfx) \\
& = \sum ((h_{(1)} k) \cdot \yfx) \circ (h_{(2)} \cdot \ \xidx) \notag.
\end{align}
\enu

In this case, we will also say that $\mA$ is a {\bf partial $H$-module  category}. 

\end{defn}

Notice that an action of $H$ on $\mA$ is clearly a partial action. On the other hand, given a partial action of $H$ on $\mA$, whenever \eqref{H4:condicion4} holds too, it is in fact an action. 
\medskip

We derive the following from ~\eqref{partialH2:condicion2} and ~\eqref{partialH3:condicion3}:
\beqna \label{useful}
h \cdot  (\zfy  \circ (k \cdot \ygx )) & = &\sum (h_{(1)} \cdot  \zfy ) \circ ((h_{(2)}k) \cdot  \ygx) 
\eqna
In fact,
\beqnast
h \cdot (\zfy \circ  (k \cdot \ygx )) & = & \sum (h_{(1)} \cdot \zfy) \circ ((h_{(2)} \cdot (k  \cdot  \ygx)) \\
& = & \sum (h_{(1)} \cdot \zfy) \circ ((h_{(2)} \cdot  \yidy ) \circ (h_{(3)}k  \cdot  \ygx)) \\
& = &\sum (h_{(1)} \cdot  (\zfy \circ  \yidy)) \circ ((h_{(2)}k) \cdot  \ygx).\\
& = &\sum (h_{(1)} \cdot  \zfy ) \circ ((h_{(2)}k) \cdot  \ygx).
\eqnast

\begin{prop} Let $H$ be a Hopf algebra and $\mathcal{A}$ be a partial $H$-module category. Then, for every $x\in \mathcal{A}_0$ the morphism space ${}_x\mathcal{A}_x$ is a partial $H$-module algebra.
\end{prop}

\begin{proof} The product in ${}_x\mathcal{A}_x$ is given by the composition, turning  this vector space into an algebra with unit ${}_x1_{x}$. The condition (\ref{partialH1:condicion1}) leads to the fact that $1_H \cdot f =f$, for every $f\in {}_x\mathcal{A}_x$. Let $f,g\in {}_x\mathcal{A}_x$ and $h,k\in H$. Then, using (\ref{partialH2:condicion2}), we have
\[
h\cdot (g\circ f)=\sum (h_{(1)}\cdot g)\circ (h_{(2)}\cdot f)
\]
and, by using (\ref{partialH3:condicion3}) we get
\[
h\cdot (k \cdot f)=\sum (h_{(1)}\cdot {}_x1_{x})\circ ((h_{(2)}k)\cdot f) .
\]
Therefore, ${}_x\mathcal{A}_x$ is a partial $H$-module algebra. Moreover, the partial action is symmetric, since (\ref{partialH3:condicion3}) also implies that 
$
h\cdot (k \cdot f)=\sum ((h_{(1)} k) \cdot f )\circ (h_{(2)}\cdot {}_x1_{x}) .
$
\end{proof}

\begin{cor} 
A partial action of a Hopf algebra $H$ on an $X$-category, with $X$ a unitary set coincides with the partial action of $H$ on the algebra defined by this category.
\end{cor}

Given a finite dimensional Hopf algebra $H$
 and a partial action of $H$ on a $k$-category $\mC$, for each pair of elements $x,y \in \mC_0$
we have a $k$-linear map
\[ H \stackrel{\Pi}\longrightarrow End_k (\ _{y}{\mathcal C}_{x}) \]
\[ h \mapsto \pi_{h}  \]
defined by $\pi_h ({}_yf_x) = h \cdot {}_yf_x$.

The conditions
\eqref{partialH1:condicion1}, \eqref{partialH2:condicion2} and \eqref{partialH3:condicion3} may be formulated in terms of $\Pi$.
Even if $\Pi$ depends on $x$ and $y$, we will just denote it by $\Pi$ to simplify the notation.   When $x = y$, we will denote  
 $\lambda_h^{x} = \pi_h \cdot \ _{x}1_{x}$. The condition \eqref{partialH3:condicion3} in the definition of partial $H$-action implies that for all $h$, $k$ in $H$
  \begin{equation}
\label{pis}  
\pi_k\pi_h= \sum_k \lambda^{x}_{k_{(1)}} \pi_{(k_{(2)}h)} = \sum_k  \pi_{(k_{(1)}h)} \lambda^{x}_{k_{(2)}}.
\end{equation} 

Note that the action is global if and only if 
$\lambda_h^{x} = \epsilon(h)  \ _{x}1_{x}$, for all $h \in H$ and for all $x \in {\mathcal C}_{0}$.

Given $x \in {\mathcal C}_{0}$, we have $\Lambda^{x} \in Hom_{k}(H, \ _{x}{\mathcal C}_{x})$ defined by $\Lambda^{x}(h) = \lambda^{x}_{h}$. Note that $Hom_{k}(H,_{x}{\mathcal C}_{x})$ is an algebra with the convolution product. It is not difficult to verify that $\Lambda^{x}$ is an idempotent and that $\Lambda^{x} (1_{H} )= \ _{x}1_{x}$.

The unit $\eta^{x}:k\to \ _{x}{\mathcal C}_{x}$ is an algebra morphism
 inducing 
$\eta^{x}_*: Hom_k(H, k) \to Hom_k(H, \ _{x}{\mathcal C}_{x})$.
Given $\Lambda^{x} = \eta^{x}_*(\widetilde{\Lambda})$ with $\widetilde{\Lambda} \in Hom_k(H, k)$, evaluating the 
equality (\ref{pis}) at $_{x}1_{x}$, we obtain that 
\begin{equation}
\lambda^{x}_k\lambda^{x}_h= \sum_k \lambda^{x}_{k_{(1)}} \lambda^{x}_{(k_{(2)}h)} 
= \sum_k  \lambda^{x}_{(k_{(1)}h)} \lambda^{x}_{k_{(2)}}
\label{eq:lambdas}.
\end{equation}

Starting with a partial action of a finite dimensional Hopf algebra $H$ 
 we have the map $\Pi$ defined as before, inducing then a $k$-linear morphism $\Lambda^x$ for each $x\in \mC_0$.
Whenever $\Lambda^x$ is of type $\Lambda^{x} = \eta^{x}_*(\widetilde{\Lambda})$ with $\widetilde{\Lambda} \in Hom_k(H, k)$ we will say that the partial action is {\bf induced by} $k$.
In particular, for a partial $H$-module category induced by $k$, we have for each $x\in \mC_0$ that the following equalities hold: 
\begin{enumerate}
\item[(a)] $\Lambda^x(1_H) = {}_x1_x$,
\item[(b)] $\Lambda^x(h) = \sum \Lambda^x(h_{(1)})\Lambda^x(h_{(2)})$,
\item[(c)] $\Lambda^x(h)\Lambda^x(k) = \sum \Lambda^x(h_{(1)})\Lambda^x(h_{(2)}k) = \sum \Lambda^x(h_{(1)}k)\Lambda^x(h_{(2)}).$
\end{enumerate}

\subsection{Partial actions of $k^G$ on $k$}\label{lambdas}

As we know, given a group $G$, a $G$-graded $k$-category $\mathcal C$ is the same as a $kG$-comodule category and 
if $G$ is finite this is the same as a $k^G$-module category (see \cite{CS}).
It follows from the axioms that for all $x \in {\mathcal C}_{0}$,  $\ _{x}1_{x}$ is homogeneous of degree $e_G$.
We have already noticed that if $\mathcal C$ is a partial $k^G$-module category such that for any $h\in k^G$, 
$h \ \cdot \ _{x}1_x = \epsilon(h) \ _{x}1_{x}$, $ \forall h \in H$ and $\forall x \in {\mathcal C}_{0}$, then ${\mathcal C}$ is in fact a $k^G$-module category.  So the first condition needed to have a partial action of $k^G$ - let's call it a partial grading - which is not global is that $_{x}1_x$ is not homogeneous of degree $e_G$ for some $x$.  

Given a 
partial $k^G$-module category induced by  $k$, where  $G$ is a finite group and $k^G$ has multiplicative basis $\{p_{g}\}_{g \in G}$,  writing $\lambda^x_g = \Lambda^x(p_g)$ for 
for each $x\in \mC_0$, the equations (a), (b) and (c) above translate (restricted to this basis) as 
\begin{enumerate}
\item[(a')] $\sum_g \lambda^x_g = 1$,
\item[(b')] $\lambda^x_g = \sum_h \lambda^x_{gh^{-1}}\lambda^x_h$,
\item[(c')] $\lambda^x_h \lambda^x_g = \lambda^x_{hg^{-1}} \lambda^x_g  =  \lambda^x_{g^{-1}h} \lambda^x_g $
\end{enumerate}
for every $g,h \in G$. 

Next we will exhibit an example of partial $H$-module category which is not global.  

\begin{exmp} \label{exemplodeparcial}
Let $H=k^{C_n}$, where $k$ is a field,  $char(k)$ does not divide $n$ and $C_n$ is the cyclic group of order $n$ with generator $t$.
As it is well-known, $H$ is a Hopf algebra with multiplicative $k$-basis $\{p_{t^i}\}_{0\le i\le n-1}$, multiplication determined by $p_{t^i}.p_{t^j}=\delta_{i,j}p_{t^{j}}$ and comultiplication 
$\Delta(p_{t^i})= \sum_{j=0}^{n-1}p_{t^{i-j}}\otimes p_{t^j}$.
Let $\mathcal{C}$ be the $k$-category with $\mathcal{C}_0=\{1,2,3\}$ and  ${}_2\mathcal{C}_1$, ${}_3\mathcal{C}_2$ and
${}_1\mathcal{C}_3$ are $k$-vector spaces with generators  $\alpha$, $\beta$ and $\gamma$, respectively. Also,  ${}_i\mathcal{C}_i$ is the $k$-vector space with generator $\ _{i}1_{i}$ ($i=1,2,3$), such that
$\beta\alpha=0=\gamma\beta=\alpha\gamma$, endowed with the following partial action of $H$ on $\mathcal{C}$:
\[ p_g.a= \frac{1}{n}a, \forall a \in \  _{x}\mathcal{C}_{y}, \forall g \in C_{n}. \]
A straightforward computation shows that the above formula defines a partial action of $H$ on $\mathcal{C}$. Since $p_g \cdot \ _{x}1_{x} \neq 0$, $\forall x \in {\mathcal C}_{0}$ and  for all $g\in C_n$, it is clear that this action is not global.
\end{exmp}

At a first glance, this partial action does not appear as a very interesting one. However,
we will deduce that there is not a lot of choices.

From equation (c') we deduce that if $\lambda_h^{x} \neq 0$ for all $h$, then $\lambda_g^{x}=\lambda_e^{x}$, for all $g\in G$, and thus, in this situation $\Lambda^{x}$ is constant. As a consequence of the equality $\Lambda^{x}(1_{k^{G}})={}_{x}1_{x}$, we get that $\lambda_g^{x}={}_{x}1_{x}/|G|$, for all $g\in G$.

This shows that the partial action of the Example \eqref{exemplodeparcial} is the only possible one if we consider partial actions induced by partial actions taking values in $k^{*}$.
Note that in general, $\lambda_e^{x}$ is a central element of ${}_{x}{\mathcal C}_{x}$.

\medskip

We summarize these  comments in the following proposition.

\begin{prop}
Given a finite group $G$ and a field $k$ such that $char(k)$ does not divide $|G|$, any partial action of $k^{G}$ on ${}_x\mC_x$, for all $x\in \mC_0$ which is induced by a partial action of $k^{G}$ on $k^{*}$, defines a unique partial action on the  $k$-category $\mathcal C$, given by $p_{g} \cdot \ _{y}f_{x} = \frac{1}{|G|}  \ _{y}f_{x}, $ for all $g \in G$, $x, y \in {\mathcal C}_{0}$, for all $_{y}f_{x} \in \ _{y}{\mathcal C}_{x}$. 
\end{prop}
\begin{proof} 
Most of the computations are straightforward.
The only thing left to prove is the equality 
$p_{g} \cdot \ _{y}f_{x} = \frac{1}{|G|}  \ _{y}f_{x}, $ for all $g \in G$, $x\neq y \in {\mathcal C}_{0}$, for all $_{y}f_{x} \in \ _{y}{\mathcal C}_{x}$.
For this,  $p_{g} \cdot \ _{y}f_{x} = p_{g} \cdot ({}_y1_y \ _{y}f_{x})= 
\sum_{h\in G} (p_{gh^{-1}}{}_y1_y)( p_h \cdot \ _{y}f_{x})= 
\sum_{h\in G} \frac{1}{|G|} p_h\cdot  \ _{y}f_{x}= \frac{1}{|G|} \ _{y}f_{x}.$ 
\end{proof}

Once this fact established, we generalize it as follows.

\begin{thm}\label{graduacoesparciais}  Let $\mC$ be a $k$-linear category and let   
$(\mC_0 \times \mC_0)^*$ denote the set of pairs $(x,y) \in \mC_0 \times \mC_0$
such that $x \neq y$ and ${}_y\mC_x  \neq 0$. Let $G$ be a finite group and 
suppose that $char(k)$ does not divide $|G|$. Given a partial $k^{G}$-action on $\mC$ induced by partial actions of $k^{G}$ on $k$, let $G_x = \{g \in G; \lambda^x_g \neq 0\}$. 
 Then \\
 (i) Each $G_x$ is a subgroup of $G$, and \\
 (ii) There is a family 
$\{{}_yt_x\}_{(x,y) \in (\mC_0 \times \mC_0)^*}$ of elements of $G$ such that ${}_y \mC_x \neq 0$ implies  $G_y = {}_yt_x \ G_x \ {}_y^{}t_x^{-1}$. \\
Moreover, if the action of $k^G$ on each ${}_y\mC_x$ is also by scalar multiplication, i.e., if each linear map  $\pi_g : {}_y\mC_x \rightarrow {}_y\mC_x  $ is a multiple of $Id_{{}_y \mC_x}$, then \\
(iii)  $({}_zt_y {}_yt_x) G_x = {}_zt_x G_x$
whenever the composition ${}_z\mC_y \otimes {}_y\mC_x \rightarrow {}_z\mC_x$ is not zero. \\
Conversely, given a pair $\left(\{G_x\}_{x \in \mC_0}, \{{}_yt_x\}_{(x,y) \in (\mC_0 \times \mC_0)^*} \right)$ satisfying (i), (ii) and (iii) one can define  a partial $k^{G}$-action on $\mC$.
\end{thm}
\begin{proof}
Suppose we have a partial $k^{G}$-action induced by $k$. For each $x\in \mC_0$ let $\Lambda^x \in Hom_k(k^G, {}_x\mC_x)$ be the associated map and let $G_x = \{g \in G| \lambda_g^x \neq 0\}$. 
It follows from the equality 
$\lambda_k^{x} \lambda_g^{x} = \lambda_k^{x}\lambda_{gk^{-1}}^{x}$ that $\lambda_{g}^{x} = \lambda_{gk^{-1}}^{x}$  for all $g, k \in G_{x}$.   In particular, $\lambda_g^x = \lambda_{gg^{-1}}^x = \lambda_e^x$, and hence 
  $\lambda_e^x \in G_x$ and 
  $\Lambda^x$ is constant on $G_x$.
Moreover the same equality can be used to prove that $G_x$ is a subgroup of $G$. 
It follows from the equality 
$\displaystyle{\sum_g \lambda_g^x = {}_{x}1_{x}}$ that there is no choice for the common value of $\lambda^x_g$ 
for  $g \in G_x$: $\lambda_g^x$ must be equal to $1/|G_x|$.

We shall see that if $\yCx \neq 0$ then the groups $G_x$ and $G_y$ are conjugated. More precisely, there exists at least one nonzero $\pi_t \in End(\yCx)$, and for every such $t$ we have $G_y = t^{-1}G_xt$. In particular, if $G$ is abelian then the family $(G_x)_{x \in \mC'_0}$, where $\mC'$ 
is a connected component of $\mC$, is constant.
In fact, $\sum \pi_g  =  Id_{\yCx}$ implies that at least one $\pi_t$ is nonzero. The third condition of a partial action yields the equations 
\begin{equation} \label{eq:pspt}
\pi_s \pi_t  =  \lambda^y_{st^{-1}} \pi_t = \lambda^x_{t^{-1}s} \pi_t 
\end{equation}
and from these equations it follows that if $\pi_t \neq 0$ and $\yfx \in \yCx$ is such that $\pi_t \cdot \yfx \neq 0$ then 
$\lambda^y_{st^{-1}} \pi_t (\yfx) = \lambda^x_{t^{-1}s} \pi_t (\yfx)$ 
and therefore $\lambda^y_{st^{-1}}  = \lambda^x_{t^{-1}s} $
for each $s \in G$. This last equation implies that $\lambda^y_{g}  = \lambda^x_{t^{-1}gt}  \text{ for every } g \in G,$ then 
$G_y = tG_x t^{-1}$.  
Of course, the $t$'s are not uniquely determined. Anyway, choosing one ${}_yt_x $ for each $(x,y) \in (\mC_0 \times \mC_0)^*$ we obtain a family
$\{{}_yt_x\}_{(x,y) \in (\mC_0 \times \mC_0)^*}$ satisfying (ii).  

\bigskip 

Suppose now that 
the partial action on each ${}_y\mC_x$ is also induced by a partial action on $k$.

If $\pi_t \neq 0$ then $\pi_{tg} \neq 0$ if and only if $g \in G_x$; in fact,  equation (\ref{eq:pspt}) implies   
\begin{equation}\label{eq:supp(yCx)}
\pi_{(tg)} \pi_{t} = \lambda^x_{tgt^{-1}} \pi_{t} 
\end{equation}
 and hence $\pi_{tg} \neq 0$ iff $tgt^{-1} \in G_y$, i.e., iff $g \in G_x$.  Therefore 
\[
\{ h \in G; \pi_h \neq 0\} = {}_yt_x \ G_x
\]
and, since $\pi_{h}$ is a multiple of $Id_{{}_y\mC_x}$ for every $h \in G$ and $(x,y) \in (\mC_0 \times \mC_0)^*$ then it follows from equation (\ref{eq:supp(yCx)}) above that ${}\pi_h = 1/|G_x|$ if $h \in {}_yt_x \ G_x$ and ${}\pi_h = 0$ otherwise.

 Since we will have to deal  with three objects in the following, given 
 $(x,y) \in (\mC_0 \times \mC_0)^*$
 we will use briefly the notation ${}^y\pi_g^x$ for the map $\pi_g  \in \End({}_y\mC_x)$. Now, equation (\ref{partialH2:condicion2}) implies that 
\[
{}^z \pi^x_g = \sum_l  {}^z \pi^x_{gl^{-1}}   {}^y \pi^x_l\]
for every $g \in G$ and every triple $(x,z),(x,y)$ and $(y,z)$ in  $ (\mC_0 \times \mC_0)^*$.
If $g \in {}_zt_x G_x$, then  
${}^z \pi^x_g = (1/|G_x|) Id_{{}_z\mC_x}$, and 
hence
\[
(1/|G_x|) Id_{{}_z\mC_x} = \sum_l  {}^z \pi^x_{gl^{-1}}   {}^y \pi^x_l = \sum_{l \in \ {}_yt_x G_x}
{}^z \pi^y_{gl^{-1}} (1/|G_x|) Id_{{}_y\mC_x} 
\]
implying that $ \displaystyle\sum_{l \in\  {}_yt_x G_x} {}^z \pi^y_{gl^{-1}} = 1$.	Since $g = ({}_zt_x \ s)$ for some $s \in G_x$ and each ${}^z \pi^y_{h}$ is equal either to $1/|Gx|$ or zero, being nonzero iff $h \in {}_zt_y \ G_x$, we conclude that 
\[\{gl^{-1}; l \in {}_yt_x G_x\} 
= \{{}_zt_x \ h \ {}_y^{}t_x^{-1}; h \in  G_x\} = {}_zt_x \ G_x,\]
 which implies that ${}_zt_x = {}_zt_y {}_yt_x h \ $ for some $h \in G_x$, and hence ${}_zt_x G_x= ({}_zt_y {}_yt_x) G_x$, proving (iii). 

\bigskip 

In order to prove the second part of the theorem, which may be considered as a kind of "converse" 
to the first one, consider a pair  $\left((G_x)_{x \in \ca_0}, ({}_yt_x)_{(x,y) \in \mcst} \right)$
satisfying (i), (ii) and (iii), and let $n$ be the common number of elements of the groups $G_x$. Consider the families of scalars 
\beqnst
\lambda_g^x  =  
\left\{
\begin{array}{l}
\frac{1}{n} \text{ \qquad if } g \in G_x \\
0 \text{ \qquad otherwise}
\end{array}
\right.
\ \ \ \ 
{}^y\pi_g^x  =  
\left\{
\begin{array}{l}
\frac{1}{n} \text{\qquad if } g \in {}_yt_x \ G_x \\
0 \text{ \qquad otherwise}
\end{array}
\right.
\eqnst
and then define the partial action by 
\beqnast
 p_g \cdot ({}_x f_x) & = & \lambda_g^x {}_x f_x ,\\
 p_g \cdot (\yfx) & = & {}^y\pi_g^x \yfx .\\ 
\eqnast
Let us show that this is indeed a partial $k^G$-action. By construction, we already have a partial $k^G$-action on each algebra $\xCx$. 

Consider a pair $(x,y) \in \mcst$. The equations that must be checked for the operators $\pi_g = {}^y\pi_g^x$, for every $g\in G$, $x$, $y\in \mC_0$ are the following: 
\beqna
\sum \pi_g & = & Id_{\yCx} \label{exampleunidim1}\\
\pi_g & =& \sum_{l} \lambda^y_{gl^{-1}} \pi_l = \sum_{l} \lambda^x_{l^{-1}g} \pi_l 
\label{exampleunidim2}\\
\pi_g \pi_h & = & \lambda^y_{gh^{-1}} \pi_h = \lambda^x_{h^{-1}g} \pi_h .
\label{exampleunidim3}
\eqna

Equations (\ref{exampleunidim1}) and (\ref{exampleunidim3}) are straightforward, and  equations (\ref{exampleunidim2})  hold because they are a consequence of 
(\ref{exampleunidim1}) and (\ref{exampleunidim3}). For instance,
\beqnst
\sum_{l} \lambda^y_{gl^{-1}} \pi_l = 
\sum_{l} \pi_g \pi_l = \pi_g (\sum \pi_l) = \pi_g.  
\eqnst
 
Conditions involving three or more distinct objects are automatically satisfied because of the hypotheses on the family $\{{}_yt_x\}$. 

\end{proof}

We remark that this result describes all structures of partial $k^G$-actions on a Schurian category (where every space of morphisms ${}_y\mC_x$ is either zero or unidimensional).

Taking a closer look at the possible idempotents $\Lambda^x$ one sees that, when $|G| \neq 0$ in $k$, these correspond to the \emph{transitive permutation representations} of $G$.

In order to explain this correspondence, consider a subgroup $H$ of $G$, of index $m$, and let $\Omega(H) =\{g_1H,g_2H,g_3H, \ldots, g_mH\}$ be the set of left cosets of $H$; take $g_1 = e$. There is a canonical left action of  $G$ on $\Omega(H) =\{H,g_2H,g_3H, \ldots, g_mH\}$ by left multiplication: $g \rhd g_i H = gg_i H$ (where $``g \rhd x''$ indicates the action of $g$ on $x$).  This action is transitive, i.e., the orbit of any element of $\Omega(H)$ is the whole set. It is easy to see that the stabilizer of the \emph{point} $H$, i.e., the group $G_H = \{g \in G; g \rhd H = H\}$ is the \emph{group} $H$ itself.

A transitive permutation representation of $G$ is a left action of $G$ on a non-empty set $\Omega$ which is transitive. Given such an action, consider an element $x_0 \in \Omega$ and its stabilizer $H =G_{x_0} = \{g \in G; g \rhd x_0 = x_0\}$. If the left cosets of $H$ are $H,g_2H,g_3H, \ldots, g_mH$,  it is easy to see that the elements of $\Omega$, which correspond to the orbit of $x_0$, are listed as $x_0,g_2x_0, \ldots, g_mx_0$ (and that their respective stabilizers are $H,g_2Hg_2^{-1}, \ldots, g_mHg_m^{-1}$). The map $\varphi : \Omega(H) \rightarrow \Omega$ defined by $g_iH \mapsto g_ix$ is an equivalence of left actions: it is bijective and $\varphi(g \rhd x) = g \rhd \varphi(x) $ for every $x \in \Omega(H)$ and every $g \in G$. Therefore, the inequivalent transitive permutation representations of $G$ are listed by the permutation representations $
\Omega(H)$ associated to subgroups $H$  of $G$.

Every permutation representation has a canonical linearization: in the case of $\Omega(H)$, we may consider the $k$-vector space generated by the left cosets $H,g_2H,g_3H, \ldots, g_mH$, and the left action of $G$ on $\Omega(H)$ gives rise to a linear representation $V_{\Omega(H)}$ of $G$ where each $g \in G$ permutes the elements of the basis $\beta = \{H,g_2H,g_3H, \ldots, g_mH\}$. 

Finally, since $|G| \neq 0$ in $k$, we may consider the idempotent $e_H = \frac{1}{|H|} \sum_{h \in H} h$ of $kG$; the left ideal $V_H$ generated by $e_H$ is the subspace of $kG$ generated by the elements $v_{H}, v_{g_1H}, \ldots, v_{g_mH}$, where $v_{g_iH} = \sum_{h \in H} g_ih$. These elements are linearly independent and hence form a basis for $V_H$. The \emph{linear representation} associated to $e_H$ is given by left multiplication on $V_H$, and the canonical bijection between the left cosets of $H$ and the vectors $v_{g_iH}$ determines an equivalence of linear representations between $V_{\Omega(H)}$ and $V_H$. Moreover, $G$ acts on $V_H$ by permuting the elements $v_{g_iH}$, and the restriction of the action of $G$ to the basis $\Omega = \{v_{g_iH}; 1 \leq i \leq m\}$ is a transitive permutation representation which is equivalent to the permutation representation $\Omega(H)$. Therefore, when $|G| \neq 0$ in $k$ the partial actions of $k^G$ on $k$ are in correspondence with the transitive permutation representations of $G$.

\medskip

Next we will give two examples of a different type.

\begin{exmp} (Partial actions of $kG$ on $k$)

The same reasoning we have used for describing partial $k^G$-actions induced by  $k$ applies for partial actions of the group algebra $kG$ induced by  $k$. Let $\{\delta_g\}_{g\in G}$ be a $k$-basis of $kG$.
Writing $\lambda_g$ for $\Lambda(\delta_g)$, the defining equations for a partial action are 
\begin{enumerate}
\item[(1)] $\lambda_e = 1$,
\item[(2)] $\lambda_g = \lambda_g^2$,
\item[(3)] $\lambda_g \lambda_h = \lambda_g \lambda_{gh}$.
\end{enumerate}
Similar calculations to those for $k^G$-partial actions yield that the support of $\Lambda$ -that is, the set  $supp(\Lambda) = \{g\in G \hbox{ such that } \lambda_g\neq 0\}$- is a subgroup  of $G$ on which $\Lambda$ is constant. By (2) we have $\lambda_h = 1$ for all $h \in supp(\Lambda)$. Again, we have a $1$-$1$ correspondence between partial $kG$-actions induced by $k$ and subgroups of $G$. 

\end{exmp}

\begin{exmp}{Partial actions of $H_4$ on $k$}
Let $H_4$ be the Sweedler 
$4$-dimensional Hopf $k$-algebra 
\[
H_4 = \langle 1,g,x,xg; g^2 = 1, x^2 =0, gx = -xg \rangle
\] 
where $char(k) \neq 2$. Consider the basis of $H_4$ formed by $e_1 = (1+g)/2$, $e_2 = (1-g)/2$, $h_1 = xe_1$, 
$h_2 = xe_2$. Then $\{e_1,e_2\}$ is a complete set of orthogonal idempotents for $H_4$ and $\{h_1,h_2\}$ is a basis of the radical of $H_4$, where $h_1^2 = h_2^2 = 0$. For the other products of pairs of elements of the basis, we have $e_1 h_2 = h_2 e_2 = h_2$, $ e_2 h_1 = h_1e_1 = h_1$, and all the remaining products of pairs of basis elements are zero.

The expressions for the coproducts in this basis are 
\begin{eqnarray}
\Delta (e_1) &=& e_1 \otimes e_1 + e_2 \otimes e_2 , \nonumber\\
\Delta (e_2) &=& e_1 \otimes e_2 + e_2 \otimes e_1 , \nonumber\\
\Delta (h_1) &=& e_1 \otimes h_1 - e_2 \otimes h_2 + h_1 \otimes e_1 
+h_2 \otimes e_2 , \nonumber\\
\Delta (h_2) &=& e_1 \otimes h_2 - e_2 \otimes h_1 + h_1 \otimes e_2 
+h_2 \otimes e_1 .\nonumber
\end{eqnarray}
The counit takes the values
$\epsilon(e_1 ) =1$ and $\epsilon (e_2 )=\epsilon (h_1 )=
\epsilon (h_2 )=0$. Finally, the antipode  on these elements is given by
\[
S(e_1 )=e_1, \quad S(e_2 )=e_2, \quad S(h_1 )=-h_2,   \quad  
S(h_2)=h_1 .
\]

Now, consider the $k$-category $\mC$ with one object $\{*\}$ and a $1$-dimensional 
$k$-vector space of morphisms. In fact, we are just looking at the algebra $k$ as a $k$-category. We want to define a partial action of $H_4$ on $\mC$.
For this, we remark that the discussion made in Subsection \ref{lambdas} about the definition of a partial action of a Hopf algebra on a $k$-category applies also for $H_4$.

Since $1_{H_4} = e_1 + e_2,$ equation \eqref{partialH1:condicion1} reads as
\beqn \label{H4parcial(a)}
\lambda_{e_1} + \lambda_{e_2} = 1.
\eqn 
Using equation \eqref{partialH2:condicion2} for $e_1$ and $e_2$, we get 
\begin{eqnarray}
\lambda_{e_1} = \lambda_{e_1}^2 + \lambda_{e_2}^2, \label{H4parcial(b)e1}\\
\lambda_{e_2} = 2 \lambda_{e_1} \lambda_{e_2}. \label{H4parcial(b)e2}
\end{eqnarray}
Therefore,  if  $\lambda_{e_2} = 0$ then $\lambda_{e_1}=1$ by (\ref{H4parcial(a)}).
Writing equation \eqref{partialH2:condicion2} for $h_1$ and $h_2$ we get 
\begin{eqnarray}
\lambda_{h_1} = 2 \lambda_{e_1} \lambda_{h_1}, \\
\lambda_{h_2} = 2 \lambda_{e_1} \lambda_{h_2}.
\end{eqnarray} 
Hence, if $\lambda_{e_1}=1$ then $\lambda_{h_1}=\lambda_{h_2}=0$, and the action is the trivial (global) action by $\varepsilon$, i.e., $\lambda_h = \varepsilon(h)$ for all $h \in H_4$.  
If
$\lambda_{e_2} \neq 0$ then $\lambda_{e_1} = \lambda_{e_2} = 1/2$ by (\ref{H4parcial(a)}) and (\ref{H4parcial(b)e2}).
Equation \eqref{partialH3:condicion3} provides 
$\lambda_{h_1}\lambda_{e_1} = 2 \lambda_{h_1} \lambda_{e_1}$
which shows that  if 
$\lambda_{e_1}= \lambda_{e_2}=1/2 $ then $\lambda_{h_1}=0$. The remaining equations put no constraints on $h_2$, and for any $\alpha \in k$ the sequence
\[
\lambda_{e_1} = \lambda_{e_2} = 1/2, \ \  \lambda_{h_1}=0,\ \  \lambda_{h_2}= \alpha 
\] 
defines a partial action of $H_4$ on $k$. This is the partial action obtained in \cite{AB2} from the dualization of a partial $H_4$-coaction on $k$ presented in \cite{CJ}. Therefore,  there are essentially two structures of partial $H_4$-module category on $\mC$ and thus on the algebra $k$.
We do not know whether different values of $\alpha$ provide isomorphic partial actions or not.

\end{exmp}

\subsection{Tensoring partial $H$-module categories} 

In this subsection we prove that given a cocommutative $k$-Hopf algebra $H$ and two partial 
$H$-module categories, their tensor product is canonically a partial $H$-module category.

So, let $H$ be a \emph{cocommutative} Hopf algebra and let $\mA$ and $\mB$ be two partial $H$-module categories. We recall that the $k$-linear category $\mA \ot \mB$ is the category whose objects are $\mA_0 \times \mB_0$, whose morphism spaces are given by 
\[ 
 \ _{(y,y')}(\mA \ot \mB)_{(x,x')} = \ _{y}\mA_x \ot \ _{y'}\mB_{x'}
\]
and composition is defined by 
\beqnast
\ _{(z,z')}(\mA \ot \mB)_{(y,y')}  \ot \ _{(y,y')}(\mA \ot \mB)_{(x,x')} 
& \rightarrow &  \ _{(z,z')}(\mA \ot \mB)_{(x,x')}\\
(\ _za_y \ot \ _{z'} b_{y'}) \ot (\ _ya_x \ot \ _{y'} b_{x'}) 
& \mapsto 
& \ _za_y \ _ya_x\ot \ _{z'} b_{y'} \ _{y'} b_{x'}  
\eqnast
   For every $(x,x') \in \mA_0 \times \mB_0$, $h \in H$ and $\ _y a_x \ot \ _{y'}b_{x'} \in \ _{(y,y')}(\mA \ot \mB)_{(x,x')}$, let 
\[
h \cdot (\ _y a_x \ot \ _{y'}b_{x'})  = 
\sum (h_{(1)} \cdot \ _y a_x) \ot (h_{(2)} \cdot \ _{y'}b_{x'}).  
\]

\begin{thm}\label{tensor_partial_H_categories}
The $k$-linear maps 
\beqnast
H \ot \ _{(y,y')}(\mA \ot \mB)_{(x,x')} & \rightarrow & \ _{(y,y')}(\mA \ot \mB)_{(x,x')} \\
h \ot\ _y a_x \ot \ _{y'}b_{x'} & \mapsto &  h \cdot (\ _y a_x \ot \ _{y'}b_{x'}) 
 \eqnast
define a structure of partial $H$-module category on $\mA \ot \mB$. 
\end{thm}

\begin{proof}
Obviously $1_H \cdot \ _y a_x \ot \ _{y'}b_{x'}  =  \ _y a_x \ot \ _{y'}b_{x'}  $ 
for every $\ _y a_x \ot \ _{y'}b_{x'} \in \ _{(y,y')}(\mA \ot \mB)_{(x,x')} $.
 For the other two properties,   
\[ h \cdot ( \ _z a_y  \ _y^{} \atil_x  \ot  \ _{z'}b_{y'} \ _{y'}\btil_{x'}) =
\sum h_{(1)} \cdot  (\ _z a_y  \ _y^{} \atil_x ) \ot  h_{(2)} \cdot ( \ _{z'}b_{y'}  \ _{y'}\btil_{x'} ) 
=\]
\[\sum (h_{(1)} \cdot  \ _z a_y )(h_{(2)} \cdot  \ _y^{} \atil_x ) \ot 
 (h_{(3)} \cdot  \ _{z'}b_{y'})(h_{(4)} \cdot  \ _{y'}\btil_{x'} )
\]
and 
\beqnast
&&\sum (h_{(1)} \cdot ( \ _z a_y  \ot  \ _{z'}b_{y'} ))(h_{(2)} \cdot ( \ _y^{} \atil_x  \ot  
\ _{y'}\btil_{x'} )) = \\
& & \sum ((h_{(1)} \cdot  \ _z a_y )\ot(h_{(2)} \cdot  \ _{z'}b_{y'})) ((h_{(3)} \cdot  \ _y^{} \atil_x )\ot(h_{(4)} \cdot  \ _{y'}\btil_{x'} )) =\\
& &  \sum (h_{(1)} \cdot  \ _z a_y )(h_{(3)} \cdot  \ _y^{} \atil_x ) \ot  (h_{(2)} \cdot \ _{z'}b_{y'} )(h_{(4)} \cdot  \ _{y'}\btil_{x'} ), 
\eqnast
and it follows that 
\[
h \cdot ( \ _z a_y  \ _y^{} \atil_x  \ot  \ _{z'}b_{y'} \ _{y'}\btil_{x'}) = \sum (h_{(1)} \cdot ( \ _z a_y  \ot  \ _{z'}b_{y'}))(h_{(1)} \cdot ( \ _y^{} \atil_x  \ot 
 \ _{y'}\btil_{x'} ))
\] when $H$ is cocommutative. Likewise, 
\beqnast
h \cdot (k \cdot ( \ _y a_x  \ot  \ _y b_x )) 
& = &\sum h \cdot ((k_{(1)} \cdot  \ _y a_x  )\ot (k_{(2)} \cdot  \ _y b_x )) \\
&= & \sum  (h_{(1)} \cdot(k_{(1)} \cdot  \ _y a_x  ))\ot (h_{(2)} \cdot(k_{(2)} \cdot  \ _y b_x ))\\
& = & 
\sum  (h_{(1)} \cdot  \ _y 1^A_y )(h_{(2)}k_{(1)} \cdot  \ _y a_x  )\ot (h_{(3)} \cdot  \ _y 1^B_y)(h_{(4)}k_{(2)} \cdot  \ _y b_x )
\eqnast
and 
\beqnast
&& 
\sum (h_{(1)} \cdot ( \ _y 1^A_y  \ot  \ _y 1^B_y))(h_{(2)}k \cdot ( \ _y a_x  \ot  \ _y b_x ))
= \\
& &
\sum ((h_{(1)} \cdot  \ _y 1^A_y ) \ot (h_{(2)} \cdot  \ _y 1^B_y) )
(h_{(3)}k_{(1)} \cdot  \ _y a_x ) \ot (h_{(4)}k_{(2)} \cdot  \ _y b_x ))= \\
& & \sum  (h_{(1)} \cdot  \ _y 1^A_y)(h_{(3)}k_{(1)} \cdot  \ _y a_x  )
\ot (h_{(2)} \cdot  \ _y 1^B_y)(h_{(4)}k_{(2)} \cdot  \ _y b_x ),
\eqnast
thus showing that 
\[
h \cdot (k \cdot ( \ _y a_x  \ot  \ _y b_x )) = 
\sum (h_{(1)} \cdot ( \ _y 1^A_y \ot  \ _y 1^B_y))(h_{(2)}k \cdot ( \ _y a_x  \ot  \ _y b_x )) \]
when $H$ is cocommutative. In the same manner we have 	
\[
h \cdot (k \cdot ( \ _y a_x  \ot  \ _y b_x )) = 
\sum (h_{(1)}k \cdot ( \ _y a_x  \ot  \ _y b_x ))(h_{(2)} \cdot ( \ _x 1^A_x \ot  \ _x 1^B_x)).
 \]
\end{proof}

\begin{cor} \label{tensor_partial_H_algebras}
Let $H$ be a cocommutative Hopf algebra, and let $A,B$ be two partial $H$-module algebras. The $k$-algebra $A \ot B$ is a partial $H$-module algebra with partial action 
\[h \cdot (a \ot b) = \sum h_{(1)} \cdot a \ot  h_{(2)} \cdot b .\] 
\end{cor}

\begin{rem} Partial $H$-module categories form a monoidal category with the trivial $H$-module category $k$ -considered as a category with one object- as its unit.
\end{rem}

\subsection{Induced partial actions}

A key example of a partial $H$-module category is obtained by restricting an action on a $k$-category to an ideal, as we explain
after the next definition. 
\begin{defn}
A {\bf central  idempotent} in a $k$-category $\mC$ is an idempotent natural transformation $e$ of the identity functor $Id_{\mC}$ to itself.
In other words, it is a collection $e = \{_{x}e_{x}\} \in \xCx$, for $x\in \mC_{0}$ of idempotents of the endomorphism algebras $\xCx$ 
verifying that for every $\yfx \in \yCx$, 
\begin{equation}
\label{central}
\yey \circ \yfx = \yfx \circ \xex.
\end{equation}
\end{defn}

Given an idempotent, we define an ideal ${\cal I}$ in $\mC$ as follows: for each $x, y \in \mC_{0}$,
\[ _{y}{\cal I}_{x} = \yey \yCx \xex = \yey \yCx = \yCx \xex. \]

As usual, if $e$ is a non trivial central idempotent, then $f = (\xidx - \xex)_{x \in \mC_{0}}$ is another non trivial central idempotent, $e$ and $f$ are orthogonal -that is, both compositions are zero-, and if $\cal I$ and $\cal J$ are respectively the ideals defined by 
$e$ and $f$ , then for all $x$ and $y\in \mC_{0}$, $\yCx$ is canonically isomorphic to $_{y}\cal{I}_{x} \oplus {}_{y}\cal{J}_{x}$. 
Every $k$-category has at least two central idempotents, corresponding respectively to the zero and the identical natural transformations.
The existence of a non trivial central idempotent in $\mC$ 
corresponds to a decomposition of each $k$-vector space of morphisms
$\yCx= \yAx \oplus \yBx$ such that 
\[
({}_za_y + {}_zb_y)\circ ({}_ya'_x + {}_yb'_x) =
{}_za_y \circ {}_ya'_x + {}_zb_y \circ {}_yb'_x 
\]
for every $x,y,z \in \mC_0$ and every $_za_y \in \zAy, {}_zb_y \in \zBy, 
{}_ya'_x  \in \yAx, {}_yb'_x \in \yBx$ and 
at least one $\yAx$ and one ${}_{y'}{\mB}_{x'}$ are not zero.

We will show in the sequel that if $\mC$ is an $H$-module category provided of a non trivial central idempotent, then the associated
ideal ${\cal I}$ has a canonical structure of partial $H$-module category. 

\begin{exmp}\label{hmodulocategoria}

From \cite{CS} we know that an $H$-module category is a $k$-category $\mC$, such that each space of morphisms is an $H$-module, 
each endomorphism algebra is an $H$-module algebra 
and the composition maps are morphisms of $H$-modules -- where
the tensor product of $H$-modules is considered as an $H$-module via the comultiplication of $H$.

Assume that $e = \{_{x}e_{x}\}_{x \in \mC_{0}}$ is a non trivial central idempotent, and let ${\cal I}$ be the ideal associated to $e$. 
Given $\yfx \in \ _{y}{\cal I}_{x}$ and $h \in H$, define  
\begin{equation}\label{induced}
h \cdot \yfx = \yey \circ (h \rhd \yfx) = (h \rhd \yfx) \circ \xex.
\end{equation}

The above formula endows ${\cal I}$ with a partial $H$-module category structure. We shall just verify conditions
~\eqref{partialH2:condicion2} and ~\eqref{partialH3:condicion3}, since ~\eqref{partialH1:condicion1} is straightforward.

For ~\eqref{partialH2:condicion2}, given $h \in H$, $\ygx \in \yIx$ and $\zgy \in \zIy$, we have 
\begin{eqnarray*}
h \cdot (\zfy  \circ  \ygx ) 
& = & \zez \circ (h \rhd  (\zfy \circ  \ygx )) \\
& = & \sum \zez \circ  (h_{(1)} \rhd \zfy)\circ  (h_{(2)}  \rhd\ygx ) \\
&\overset{\text{(\ref{central})}}{=} & \sum \zez \circ  (h_{(1)} \rhd \zfy) \circ  \yey \circ  (h_{(2)}  \rhd \ygx ) \\
& = & \sum (h_{(1)} \cdot \zfy) \circ  (h_{(2)} \cdot \ygx).
\end{eqnarray*}
For the first equality of ~\eqref{partialH3:condicion3}, given $h,k \in H$ and $\yfx \in \yIx$,  
\begin{eqnarray*}
h \cdot (k \cdot  \yfx ) 
& = & \yey \circ  (h \rhd (\yey \circ (k \rhd \yfx ))) \\
& = & \sum \yey \circ  (h_{(1)} \rhd \yey)\circ  (h_{(2)}k \rhd \yfx ) \\
&\overset{\text{(\ref{central})}}{=}  & \sum \yey \circ (h_{(1)} \rhd \yey) \circ \yey \circ (h_{(2)}k \rhd \yfx ) \\
& = & \sum (h_{(1)} \cdot  \yey )\circ (h_{(2)} k \cdot \yfx)
\end{eqnarray*}
The other equality follows analogously, using this time that $h \cdot \yfx = (h \rhd \yfx) \circ \xex$. 
Hence ${\cal I}$ is, as we claimed, a partial $H$-module category. 
\end{exmp}

\subsection{Partial group actions and partial $kG$-module categories}\label{ejemploideal}

Given a group $G$, we want now to establish a relation between the concepts of partial $kG$-module category and of being a $k$-linear category
with a partial $G$-action, as defined in \cite{CFM}.
So we start by recalling the latter.

Let $\mC$ be a $k$-linear category. The group $G$ acts partially on $\mC$ if $G$ acts on the set $\mC_{0}$ and for every $g \in G$, 
there exists an ideal ${\cal I}^{g} \unlhd \mC$ 
and a isomorphism of $k$-semicategories 
\[\alpha_{g}:{\cal I}^{g^{-1}} \rightarrow {\cal I}^{g}\]  
satisfying the following conditions:
\begin{enumerate}
\item[(a)] ${\cal I}^{e} = \mC$ and $\alpha_{e} = Id_{\mC}$.
\item[(b)] Given $x,y \in \mC_0$ and $g,h \in G$,
\[\alpha_{h^{-1}}(_{hy}{\cal I}_{hx}^{h} \cap _{hy}{\cal I}_{hx}^{g^{-1}}) \subseteq _{y}{\cal I}_{x}^{(gh)^{-1}}.\]
\item[(c)] For every $_{y}f_{x} \in \alpha_{h^{-1}}(_{hy}{\cal I}_{hx}^{h} \cap _{hy}{\cal I}_{hx}^{g^{-1}})$, we have $\alpha_{g} \circ \alpha_{h} (_{y}f_{x}) = \alpha_{gh}(_{y}f_{x})$.
\end{enumerate}

We shall see next that every partial $kG$-module category $\mC$ has a canonical partial $G$-action. 

\begin{prop} Let $G$ be a group and $\mC$ a $k$-category. There is a one to one correspondence between partial $kG$-module category structures on $\mC$ and partial $G$-actions on $\mC$ fixing objects such that the ideals are generated by idempotents. 
\end{prop}
\begin{proof}
Suppose that $\mC$ is a partial $kG$ module category. Given $g\in G$, for every $x \in \mC_{0}$,  we define an idempotent $\xegx \in \xCx$ by
\[
\xegx = \ug \cdot \xidx.
\]

The fact that $\xegx$ is indeed an idempotent follows from equation ~\eqref{partialH2:condicion2}.

The family $e^g = \{\xegx\}_{x \in \mC_0}$ is central since, for each $\back \yfx \in \back \yAx$
\beqnast
\yegy \circ \yfx 
& = & (g \cdot \yidy) \circ \yfx  =  (g \cdot \yidy) \circ (gg^{-1} \circ \yfx) 
 =  (g \cdot (\yidy \circ (g^{-1} \circ \yfx)) \\
 & = & (g \cdot ((g^{-1} \circ \yfx)\circ \xidx) =  \yfx \circ (g \cdot \xidx) \\
& = & \yfx \circ \xegx.
\eqnast
  
Thus we may take $\ci^g$ as the ideal generated by $e^g$, i.e., $(\ci^g)_{0}=\mC_{0}$ and 
\[
{}_{y}{\cal I}_{x}^{g} = \yegy \yCx \xegx.
\] 
The isomorphism  
$\alpha_{g}: {\cal I}^{g^{-1}} \rightarrow {\cal I}^{g}$
is defined by $\alpha_g (x)=x$, for every $x\in (\ci^g)_{0}$ and 
$\alpha_g(\yfx) = \ug \cdot \yfx$ for $\back \yfx \in \ _{y}{\cal I}_{x}^{g^{-1}}$; the fact that this is indeed an isomorphism follows from the definitions of $\back \xegx$ and of partial $kG$-module category. 

Condition (a) is obviously satisfied. Since $G$ acts trivially on the objects, (b) may be rewritten as 
\[\alpha_{h^{-1}}(_{y}{\cal I}_{x}^{h} \cap \ _{y}{\cal I}_{x}^{g^{-1}}) \subseteq \ _{y}{\cal I}_{x}^{(gh)^{-1}}.\]
Given $x,y \in \mC_{0}$, $g,h \in G$ and $\back \yfx \in \ _{y}{\cal I}_{x}^{h} \cap \ _{y}{\cal I}_{x}^{g^{-1}}$, 
we have that $\yfx = \yfx \ _x^{}e_x^{g^{-1}}$, and therefore
\beqnast
\alpha_{h^{-1}}(\yfx) 
& = & \alpha_{h^{-1}}(\yfx \ _x^{}e_x^{g^{-1}})
 =  h^{-1} \cdot (\yfx (g^{-1} \cdot \xidx )) 
 =  (h^{-1} \cdot \yfx )(h^{-1} \cdot (g^{-1} \cdot \xidx )) 
 \\
& = & (h^{-1} \cdot \yfx )(h^{-1} \cdot \xidx)(h^{-1}g^{-1} \cdot \xidx )
 =  (h^{-1} \cdot \yfx )((gh)^{-1} \cdot \xidx ) 
 \\
& = & (h^{-1} \cdot \yfx )\ _x^{}e_x^{(gh)^{-1}}
\eqnast
which lies in $_y \ci^{(gh)^{-1}}_x$. Finally, consider elements of 
$\alpha_{h^{-1}}(_{y}{\cal I}_{x}^{h} \cap \ _{y}{\cal I}_{x}^{g^{-1}})$. They are of the 
form $(h^{-1} \cdot \yfx) \ _xe_x^{(gh)^{-1}}$, where $\yfx \in\ _{y}{\cal I}_{x}^{h} \cap \ _{y}{\cal I}_{x}^{g^{-1}}$, 
and it can be shown that 
\[
\alpha_{gh}((h^{-1} \cdot \yfx) \ _x^{}e_x^{(gh)^{-1}})
= (g \cdot \yfx) \ _x^{}e_x^{gh} =\alpha_g (\alpha_h((h^{-1} \cdot \yfx)\ _x^{}e_x^{(gh)^{-1}}
\]

Therefore, the data $(\{ {\cal I}^{g} \}_{g \in G}, \{ \alpha_{g} \}_{g \in G} )$ define a partial action of $G$ on $\mC$. 

On the other hand, if $(\{ {\cal I}^{g} \}_{g \in G}, \{ \alpha_{g} \}_{g \in G} )$ provide a partial action of $G$ on $\mC$, such that $G$ acts trivially on $\mC_{0}$ and every ideal $\ci^g$ is  generated by a central idempotent  
$e^g = \{ \back \xegx \}_{x \in \mC_{0}}$, then  the linear maps 
\beqnast
kG \ot \yCx & \rightarrow & \yCx \\
g \ot \yfx & \mapsto & \ug \cdot \yfx = \alpha_{g} (\yfx \circ  \xegx ) 
\eqnast
define a  partial action of $kG$ on $\mC$.
\end{proof}

\section{Globalization of partial $H$-actions}\label{globalization}

Globalization of partial actions first appeared in the context of partial group actions 
on $C^*$ algebras \cite{A}, then a totally algebraic formulation for the group case 
appeared in \cite{DE}. For partial actions of Hopf algebras, the globalization theorem was 
proved in \cite{AB1}. Basically, it states that every partial action of a Hopf algebra $H$ on a unital algebra $A$ admits a globalization. A globalization, or an enveloping action, means that there is a pair $(B,\varphi)$, where $B$ is an $H$-module algebra and $\varphi :A\rightarrow B$ 
is a monomorphism of algebras such that $\varphi (A)$ is an ideal in $B$ and the partial action 
on $A$ can be viewed as an induced partial action on the ideal $\varphi (A)$. In this section 
we are going to prove that a partial $H$-module category $\mC$ always has a globalization. 
Roughly speaking, we will follow the same ideas as in the globalization for partial $H$-module algebras: we will define  an 
$\mC_{0}$-semicategory $\mB$, which is an $H$-module (semi)category and a faithful
semifunctor $F:\mC \rightarrow \mB$ such that the partial action of $H$ on $\mC$ can be thought as 
a restriction of the global action $\mB$. 
Moreover, we will prove the existence of a minimal globalization, which is unique up to isomorphism.

So we start by giving the definition of globalization.

\begin{defn} Let $\mC$ be a partial $H$-module category. A {\bf globalization} 
of the partial action is a pair $(\mB, F)$ where 
\benu
\item[(a)] $\mB$ is an $H$-module semicategory over $\mC_{0}$, with action $\rhd$.
\item[(b)] $F: \mC \rightarrow \mB$ is a faithful $\mC_{0}$-semifunctor and $F(\mC)$ is an ideal of $\mB$, generated by the central idempotent  $e = \{F(\xidx)\}_{x \in \mC_0}$.
\item[(c)] $\mB = H \rhd F(\mC)$.
\item[(d)] $F$ intertwines the partial action on $\mC$ and the induced partial action 
on $F(\mC)$: that is, for every $\yfx \in \yCx$, we have
\[F(h \cdot  \yfx ) = F( \yidy) (h \rhd F( \yfx)) = (h \rhd F( \yfx)) F( \xidx).\] 
\enu
Moreover, a globalization $(\mB ,F)$ is said to be minimal if, in addition, it satisfies
\benu
\item[(e)] For every pair $x,y\in \mC_0$, if we have 
\[
\sum_{i=1}^n kh_i \cdot {}_yf^{i}_x =0, \qquad \forall k\in H,
\]
then
\[
\sum_{i=1}^n h_i \rhd F({}_yf^{i}_x)=0
\]
\enu
\end{defn}

This definition is motivated by the definition of an enveloping action for a partial $H$-action 
on a unital algebra $A$ given in \cite{AB1} and generalizes it in the following sense.

\begin{prop} Let $\mC$ be a partial $H$-module category and $(\mB ,F)$ be a 
(minimal) globalization for this partial action. Then, for every $x\in \mC_0$, the 
algebra ${}_x\mB_x$ with the morphism $F: {}_x\mC_x \rightarrow {}_x\mB_x$ is a 
(minimal) globalization for the partial action of $H$ on ${}_x\mC_x$.
\end{prop}
\begin{proof} As $\mB$ is an $H$-module semicategory, then ${}_x\mB_x$ is automatically 
an $H$-module algebra. As $F:\mC \rightarrow \mB$ is a faithful semifunctor, then, by 
definition, its restriction to ${}_x\mC_x$ is an algebra monomorphism. The fact that $F(\mC)$ 
is an ideal in the semicategory $\mB$ implies, also by definition, that $F({}_x\mC_x)$ is an ideal 
in ${}_x\mB_x$. Since $\mB =H\rhd F(\mC)$, the morphism space ${}_x\mB_x$ coincides with
$H\rhd F({}_x\mC_x)$. Let $f\in {}_x\mC_x$ and $h\in H$, then, by the item (d), we have
\[
F(h\cdot f)=F({}_x1_x)(h\rhd F(f)) ,
\]
which is the exact expression that intertwines the partial action of $H$ on ${}_x\mC_x$ and 
the induced partial action on $F({}_x\mC_x)$. Therefore, $({}_x\mB_x ,F)$ is a globalization 
for the partial action of $H$ on the algebra ${}_x\mC_x$.

Finally, if the globalization $(\mB ,F)$ is minimal, then (e) is valid for any pair 
$x,y\in \mC_0$, in particular for $y=x$, which gives the minimality condition for $({}_x\mB_x ,F)$.
\end{proof}

\begin{cor} If $\mC$ is a partial $H$-module category and $\mC_0$ is a unitary set, then any
globalization $(\mB ,F)$ coincides with a globalization of the partial action of $H$ on the 
unital algebra defined by the category $\mC$.
\end{cor}

Some remarks about the definition above are necessary. Item (a) has no analogue for 
partial actions on algebras, it appears only in the categorical context. Item (e) is 
important because only in the minimal case one can prove uniqueness of the globalization up 
to isomorphism.     

We shall prove that any partial $H$-module category admits a globalization satisfying the
conditions of the definition above, following the same steps used in \cite{AB1}. We shall 
first give a definition.

\begin{defn}
Given a Hopf $k$-algebra $H$ and a $k$-linear category $\mC$, the category $\Hom_{k}(H, \mC)$ is defined as follows:
\begin{enumerate}
\item[(a)] Its objects are the objects of $\mC$.
\item[(b)] Given $x$ and $y\in \mC_0$, let $_{y}\Hom_{k}(H, \mC)_{x} := \Hom_{k}(H,\yCx)$, with
composition given by the convolution product
\[_{z}\Hom_{k}(H, \mC)_{y} \otimes \ _{y}\Hom_{k}(H, \mC)_{x} \rightarrow \ _{z}\Hom_{k}(H, \mC)_{x}\]
\[(\zfy *  \ygx)(h) = \sum  \zfy(h_{(1)}) \circ  \ygx(h_{(2)}).\]
\item[(c)] The action of $H$ on $\Hom_{k}(H, \mC)$ is: 
\[(h \rhd  \yfx)(k) =  \yfx(kh).\]
\end{enumerate}
\end{defn}

The aim of this section is to prove that any partial $H$-module category has a minimal globalization.

\begin{thm}\label{minimal globalization} 
Every  partial $H$-module category $\mC$ has a minimal globalization, which is unique up to isomorphism.
\end{thm}

\begin{proof}
Let $\mC$ be a partial $H$-module category with partial action denoted by a dot ``$\cdot$''. 
The previous formulas provide the $\mC_0$-category $\Hom_{k}(H, \mC)$ of the structure of and $H$-module category.

We next define the $\mC_0$-semifunctor $F: \mC \rightarrow \Hom_{k}(H, \mC)$ by $F(_{y}f_{x})(h) = h \cdot \ _{y}f_{x}$ for
every $x,y \in \mC_0$, $\yfx \in  \yCx$ and $h \in H$.  To verify that $F$ is, indeed, a semifunctor, take $x,y,z\in \mC_0$, ${}_zf_y\in {}_z\mC_y$, ${}_yg_x\in {}_y\mC_x$ and $h\in H$, then 
\begin{eqnarray*}
F({}_zf_y)*F({}_yg_x)(h) &=& \sum F({}_zf_y)(h_{(1)} \circ F({}_yg_x)(h_{(2)})
= \sum (h_{(1)} \cdot {}_zf_y)\circ (h_{(2)}\cdot {}_yg_x) \\
&=& h\cdot ({}_zf_y \circ {}_yg_x) = F({}_zf_y \circ {}_yg_x)(h). 
\end{eqnarray*}

The semifunctor $F$ is also faithful. Indeed, let $x,y\in \mC_0$ and ${}_yf_x \in {}_y \mC_x$ such that $F({}_yf_x)=0$. Then, for every $h\in H$, we have
\[
0=F({}_yf_x)(h)=h\cdot {}_yf_x .
\]
In particular, putting $h=1_H$, we obtain ${}_yf_x =0$.

Consider now the $\mC_0$-subsemicategory $\mB$ of $\Hom_{k}(H, \mC)$ such that for $x$ and $y\in \mC_0$,
\begin{eqnarray*}
\yBx & = &  H \rhd F(\yCx) \\
& = &  \langle h \rhd F(\ _{y}f_{x})| h \in H, \ _{y}f_{x} \in \yCx \rangle_{k}.
\end{eqnarray*}
In order to see that $F(\mC)$ is an ideal of $\mB$, 
note first that  given $h \in H$, $x$, $y$ and $z \in \mC_0$, $\back \ygx \in \yCx$ and $\back \zfy \in \zCy$, 
we have
\beqna
F(\zfy)(h \rhd F(\ygx)) & = &  F(\zfy(h \cdot \ygx)) \label{idealdir}\\
(h \rhd F(\zfy))F(\ygx) & = &  F((h \cdot \zfy)\ygx) \label{idealesq}
\eqna
since
\beqnast
F(\zfy))(h \rhd F(\ygx))(k) & = & \sum F(\zfy))(k_{(1)})(h \rhd F(\ygx))(k_{(2)}) 
 =  \sum (k_{(1)} \cdot \zfy)(k_{(2)}h \cdot \ygx)
 \\
&\overset{\text{(\ref{useful})}}{=} & (k \cdot \zfy (h \cdot \ygx)) 
=  F(\zfy (h \cdot \ygx))(k). 
 \eqnast
The proof of the second equality is analogous.

Now, composing generators  of $\back \zBy$ and $\back \yBx$, we get 
\beqnast
(h \rhd F(\zfy))(k \rhd F(\ygx)) 
& = & \sum (h_{(1)} \rhd F(\zfy))(\varepsilon(h_{(2)}) k \rhd F(\ygx))\\ 
& = & \sum (h_{(1)} \rhd F(\zfy))(h_{(2)} \rhd S(h_{(3)}) k \rhd F(\ygx)) \\
& = & \sum (h_{(1)} \rhd [ F(\zfy)(S(h_{(2)}) k \rhd F(\ygx))]  \\
& = & \sum (h_{(1)} \rhd F(\zfy (S(h_{(2)}) k \cdot \ygx )),
\eqnast
which lies in $\back \zBx$, and therefore $\mB$ is actually a subsemicategory of $\Hom_k(H,\mC)$.  

It follows from equations (\ref{idealdir}) and (\ref{idealesq}) that if we define $\back \xex = F(\xidx)$ 
for each $x \in \mB_0$, then $e = \{\xex\}_{x \in \mB_0}$
is a central idempotent of $\mB$ and $F(\mC)$ is the ideal generated by $e$. 
Since
\[
F(h \cdot \yfx) = F(\yidy) (h \rhd F( \yfx)) = (h \rhd F( \yfx))F( \xidx) 
\]
for every $\back \yfx \in \yCx$ and $h \in H$, the pair $(\mB, F)$ is a globalization of 
the partial $H$-module category $\mC$.

In addition, it is possible to prove that $(\mB ,F)$ as constructed above is actually a minimal globalization. Indeed, let $x,y\in \mC_0$, $h_1, \ldots , h_n \in H$ and $f^1, \ldots ,f^n \in {}_y\mC_x$ such that, for every $k\in H$
\[
\sum_{i=1}^n kh_i\cdot f^i =0 .
\] 
Then, given $k\in H$, we have
\[
\left( \sum_{i=1}^n h_i\rhd F(f^i)\right) (k) 
=  \sum_{i=1}^n (h_i\rhd F(f^i) ) (k) 
=  \sum_{i=1}^n F(f^i) (kh_i )
=  \sum_{i=1}^n kh_i\cdot f^i =0 
\]
which leads to the conclusion that $\sum_{i=1}^n h_i\rhd F(f^i)=0$, proving that this globalization satisfies the minimality condition.

For the uniqueness, consider two minimal globalizations of $\mC$, $(\mA ,F)$ and 
$(\mB , G)$, where the global actions of $H$ on $\mA$ and on $\mB$ are denoted, respectively, by $\triangleright$ and $\blacktriangleright$. Define the $\mC_0$ semifunctor $\Phi :\mA \rightarrow  \mB$ such that for any pair $x,y\in \mC_0$ 
\[
\Phi \left( \sum_{i=1}^n h_i \triangleright F({}_yf^i_x) \right) =\sum_{i=1}^n h_i \blacktriangleright G({}_yf^i_x) .
\]
This functor is given in this form because $\mA =H\triangleright F(\mC )$ and $\mB =H\blacktriangleright G(\mC )$. First, one needs to prove that the maps $\Phi: {}_y\mA_x \rightarrow {}_y\mB_x$ above are well defined, as $k$-linear morphisms, for every pair $x,y\in \mC_0$, which corresponds to prove that if $\sum_{i=1}^n h_i \triangleright F({}_yf^i_x) =0$ then  $\sum_{i=1}^n h_i \blacktriangleright G({}_yf^i_x) =0$. 

 Fix $x,y\in \mC_0$ and take elements $h_1 ,\ldots ,h_n \in H$ and $f^1 ,\ldots ,f^n \in {}_y\mA_x$ such that 
\[
X=\sum_{i=1}^n h_i \triangleright F(f^i) =0
\]
Then, for all $k\in H$ we have
\[
0=F({}_y1_y )(k\triangleright X)=F({}_y1_y )\left( \sum_{i=1}^n kh_i \triangleright F(f^i) \right) =F\left( \sum_{i=1}^n kh_i \cdot f^i \right) .
\]
As $F$ is a faithful functor, we obtain that 
\[
\sum_{i=1}^n kh_i \cdot f^i =0 ,
\]
for any $k\in H$. And finally, as $(\mB ,G)$ is a minimal, this implies that
\[
\sum_{i=1}^n h_i \blacktriangleright G(f^i) =0 .
\]

By construction, $\Phi$ automatically intertwins the global actions of $H$ on $\mA$ and on $\mB$. In order to verify that $\Phi$ is, indeed, a semifunctor, consider $x,y,z\in \mC_0$, ${}_zf_y \in {}_z\mA_y$, ${}_yg_x \in  {}_y\mA_x$ and $h,k\in H$. Then
\begin{eqnarray}
\Phi ((h\triangleright F({}_zf_y ))\circ (k\triangleright F({}_yg_x )) &=& \Phi \left( \sum (h_{(1)}\triangleright F({}_zf_y ))\circ (h_{(2)}S(h_{(3)})k\triangleright F({}_yg_x )) \right)\nonumber\\
&=&  \Phi \left( \sum h_{(1)}\triangleright 
(F({}_zf_y )\circ (S(h_{(2)})k\triangleright F({}_yg_x ))) \right)\nonumber\\
&=&  \Phi \left( \sum h_{(1)}\triangleright 
F({}_zf_y \circ (S(h_{(2)})k\cdot {}_yg_x )) \right)\nonumber\\
&=&   \sum h_{(1)}\blacktriangleright 
G({}_zf_y \circ (S(h_{(2)})k\cdot {}_yg_x )) \nonumber\\
&=&   \sum h_{(1)}\blacktriangleright (
G({}_zf_y) \circ G(S(h_{(2)})k\cdot {}_yg_x )) \nonumber\\
&=&   \sum h_{(1)}\blacktriangleright (
G({}_zf_y) \circ (S(h_{(2)})k \blacktriangleright G({}_yg_x ))) \nonumber\\
&=&   \sum (h_{(1)}\blacktriangleright 
G({}_zf_y)) \circ (h_{(2)} \blacktriangleright (S(h_{(3)})k \blacktriangleright G({}_yg_x ))) \nonumber\\
&=&   (h\blacktriangleright 
G({}_zf_y)) \circ (k \blacktriangleright G({}_yg_x )) \nonumber\\
&=&  \Phi (h\triangleright F({}_zf_y ))\circ 
\Phi (k\triangleright F( {}_yg_x )) . \nonumber
\end{eqnarray}
By a completely analogue argument, on can construct a semifunctor $\Psi :\mB \rightarrow \mA$ fixing objects and for all $x,y\in \mC_0$, $\Psi :{}_y\mB_x \rightarrow {}_y\mA_x$ is given by
\[
\Psi \left( \sum_{i=1}^n h_i \blacktriangleright G({}_yf^i_x) \right) =\sum_{i=1}^n h_i \triangleright F({}_yf^i_x) .
\]
it is trivial to prove that the semifunctors $\Phi$ and $\Psi$ are mutually inverses, therefore these two minimal globalizations of the partial action of $H$ on $\mC$ are isomorphic. 
\end{proof}

\subsection{Globalization of the partial actions of $k^G$ on $k$}

Let $\Lambda_H \in \Hom(k^G,k)$ be the idempotent associated to a partial action of $k^G$  induced by $k$. It is well-known that 
\beqnast
\Theta: \Hom(k^G,k) & \rightarrow & k G \\
f & \mapsto & \sum_g f(g) \delta_g
\eqnast
is an isomorphism of Hopf algebras.
We will identify $\Hom(k^G,k)$ and  $k G$  via this isomorphism.

Following the proof of the globalization theorem, consider the composite map 
\[
k \overset{\varphi}{\longrightarrow} \Hom(k^G,k)  \overset{\Theta}{\longrightarrow}  k G 
\]
where $\varphi$ is the map that takes $1 \in k$ to the map $ \varphi(1)(h) = h \cdot 1$, i.e., $\varphi(1) = \Lambda$.  Therefore $(\Theta \circ \varphi)(1) = \Theta (\Lambda) = \frac{1}{|H|} \sum_{h \in H} \delta_h $, an idempotent which we will call $e_H$.

The canonical action of $k^G$ on $\Hom(k^G,k)$ corresponds, via $\Theta$, to the action of  $k^ G$ on  $kG$ associated to the $G$-grading of $k^G$: $p_g \rhd \sum a_s \delta_s = a_g$. The subalgebra $B$ of $kG$ that contains $k^G$ as an ideal is generated by the elements $p_g \rhd e_H$, with $g \in G$, and 
\[p_g \rhd e_H = 
\left\{
\begin{array}{ll}
\small{\frac{1}{|H|}}\delta_g \text{ if }g \in H \\
0 \text{ otherwise}
\end{array}
\right.
\] 
Therefore, the globalization $B$ is the $k^G$-module algebra $kH \subset kG$. 

\bigskip

\section{Partial smash products of categories}\label{partial smash}

In this section, we generalize the construction of partial smash products, as defined in \cite{CJ} to the context of partial $H$ module categories. 

First, let us make a brief recall of the construction of partial smash products for algebras. Let $H$ be a Hopf algebra acting partially on a unital algebra $A$. One can define a product on the $k$ vector space $A\otimes H$, given by
\[
(a\otimes h)(b\otimes k)=\sum a(h_{(1)}\cdot b)\otimes h_{(2)}k.
\]
This product is automatically associative because of the axioms of partial actions. It is also easy to see that the element $1_A \otimes 1_H$ is a left unit with respect to this product. In order to obtain a unital algebra, one needs to take the right ideal
\[
\underline{A\# H}=(A\otimes H)(1_A \otimes 1_H) =\{ \sum a(h_{(1)}\cdot 1_A)\otimes h_{(2)} \, | \, a\in A ,\; h\in H \} .
\]
This is the so-called partial smash product. In the particular case of $H$ being a group algebra $kG$, this partial smash product coincides with the partial skew group ring, as defined in \cite{DE}.

Next, we will show that just as in the case of partial $H$-module algebras, we can define a smash product for partial $H$-module categories which, of course, generalizes also the smash product construction for categories \cite{CS}. Following the same steps as in 
\cite{CJ}, we begin with an intermediate semicategory. 

Let $\mA$ be a partial $H$-module partial category. Consider the $\mA_0$-semicategory 
$\mA \otu H$ which,  for every $x, y \in \mA_0$, has morphism space defined by 
\[
{}_{y}(A \otu H)_{x} = \ _{y}A_{x} \otimes H
\]
with composition given by
\[
(_{z}f_{y} \otimes h) \circ (_{y}g_{x} \otimes h') = \sum \zfy \circ 
(h_{(1)} \cdot \ygx ) \otimes h_{(2)}h'.
\]

The axioms of partial $H$-module category guarantee  that composition is well-defined and associative. 

In order to get an $\mA_0$-category with "local" units, which we will use to construct a globalization, we consider a right ideal in $\mA \otu H$. Note that the collection $e = \{\xidx \ot 1_H\}_{x \in \mA_0}$ is not central. However, it satisfies 
\beqna \label{unitsmash}
(\yidy \ot 1_H)(\yfx \ot h) & = & (\yfx \ot h)
\eqna
and, in particular, each morphism $\back \xidx \ot 1_H$ is an idempotent in the algebra $_x(\mA \otu H)_x$.

\begin{defn}
The partial smash  $H$-module category $\underline{\mA \# H}$ is the $\mA_0$-category   defined by 
\beqnast
\yASHx 
& = &\ _{y}(\mA \otu H)_{x} \circ (\xidx \otimes 1_{H}) \\
& = & \langle {}_{y}f_{x}\# h \mid \yfx \in \yAx, h \in H \rangle_k
\eqnast 
where ${}_{y}f_{x}\# h  = \sum \ _{y}f_{x} \circ (h_{(1)} \cdot _{x}1_{x}) \otimes h_{(2)} $ and the composition is the same one of $\mA \otu H$. 
\end{defn}

\begin{rem} Note that $\underline{\mA \# H}$ is indeed a category: by definition, it is the right ideal in $\mA \otu H$ generated by $e$, and hence composition is well-defined and associative. From equation (\ref{unitsmash}) it follows that the element 
$\back \xidx \ot 1_H$ is the identity of $\back \xASHx$, for each $x$.
\end{rem} 

In \cite{AB1} it was proved that if $(B\varphi)$ is a globalization of a partial action of a Hopf algebra $H$ on a unital algebra $A$, then there is an algebra monomorphism 
$\Phi :\underline{A\# H} \rightarrow B\#H$ given by
\[
\Phi (a\# h)=\sum \varphi (a) (h_{(1)}\rhd \varphi (b))\otimes h_{(2)} .
\]
By an analogue construction to the case of algebras, if $(\mB, F)$ is a globalization of a partial $H$-module category $\mA$, it is straightforward to verify that there is a faithful semifunctor from the partial smash product category $\ASH$ to the smash product $\mB \# H$. 

\begin{rem}
if $H$ is finite dimensional, then there is a canonical \emph{global} action of $H^*$ on $\underline{\mA \# H}$.
\end{rem}

In fact, $\mA \otimes H$ has  the canonical right $H$-comodule structure, as defined in \cite{CS}, given by the coaction 
$\rho: \mA \otimes H \rightarrow (\mA \otimes H )\otimes H$, which, for all 
$x,y\in \mA_0$ sends $\ _yf_x \otimes h  $ to $\sum \yfx \otimes h_{(1)} \otimes h_{(2)}$. Since the idempotent $\ _xe_x =  \xidx \otimes 1_H $ is a co-invariant, if $f \ _xe_x \in \underline{ \ _y\mA_x \# H}$ then 
\[
\rho(f \ _x e_x) = \rho (f) \rho( {}_x e_x) = \rho(f)( {}_x e_x \otimes 1) \in 
  \underline{{}_y\mA_x \# H} \otimes H
\] 
and  $\underline{\mA \# H}$ inherits the $H$-comodule structure induced by the restriction of $\rho$. All properties of $H$-comodule category are easily checked; for instance, 
\[
(\yidy \otimes \Delta)\rho (f \ _xe_x) = (\yidy \otimes \Delta)\rho (f) \rho ( \ _xe_x)
=
(\rho \otimes I_H)\rho (f) \rho ( \ _xe_x) = 
(\rho \otimes I_H)\rho (f \ _xe_x) . 
\]

When $H$ is finite-dimensional there is a correspondence between right  $H$-comodule categories and left $H^*$-module categories. In this particular case, the $H^*$-action is defined by 
\beqnast
 H^* \otimes \underline{ \ _y\mA_x \# H} & \rightarrow &  \underline{ \ _y\mA_x \# H} \\
 \varphi \otimes ( \ _yf_x \otimes h) & \mapsto & \sum \ _yf_x \otimes h_{(1)} \varphi (h_{(2)}) 
\eqnast

Now, consider a partial action of a Hopf algebra $H$ on a $X$-category $\mA$, where $X$ is a finite set. As in example \ref{example_X_finite}, we may consider the matrix algebra over $\mA$ 
\[
a(\mA )=\left\{ \left. ({}_yf_x )_{y,x} \, \right| \, {}_yf_x \in {}_y\mA_x \right\} .
\]
It is also possible to define the matrix algebra over the partial smash product category
\[
a(\underline{\mA \# H}) =\left\{ \left. ({}_yf_x\# h^{y,x})_{y,x} \, \right| \, {}_yf_x \in {}_y\mA_x ,\; h^{y,x}\in H \right\} .
\]

\begin{prop} There is a partial action of $H$ on the algebra $a(\mA )$ such that the partial smash product $\underline{a(\mA )\# H}$ is isomorphic to the matrix algebra $a(\underline{\mA \# H})$.
\end{prop} 

\begin{proof} Define the map $\bullet :H\otimes a(\mA)\rightarrow a(\mA )$ given by 
\[
h\bullet \left( {}_yf_x \right)_{y,x} =\left( h\cdot {}_yf_x \right)_{y,x} ,
\]
which is clearly linear. Also obvious is the fact that, for any matrix $A\in a(\mA)$, we have $1_H\bullet A=A$. Now, take two matrices $A=\left( {}_yf_x \right)_{y,x}$ and $B=\left( {}_yf_x \right)_{y,x}$ and $h\in H$, then
\begin{eqnarray}
h\bullet AB &=& \left( h\cdot (\sum_{z\in \mA_0} {}_yf_z \circ {}_zg_x) \right)_{y,x}\nonumber\\
&=& \left( \sum_{z\in \mA_0} \sum (h_{(1)}\cdot  {}_yf_z ) \circ (h_{(2)} \cdot {}_zg_x) \right)_{y,x} \nonumber\\
&=& \sum (h_{(1)}\bullet A)(h_{(2)}\bullet B)\nonumber
\end{eqnarray}
The composition law can be also verified for any $h,k\in H$: recalling that the unity of $a(\mA)$ is 
$1_{a(\mA)} = {}_{x_1}1_{x_1}E_{1,1} + \cdots + {}_{x_1}1_{x_1}E_{1,1}$, we have  
\begin{eqnarray}
h\bullet (k\bullet A) &=& \left( h\cdot ( k\cdot {}_yf_x ) \right)_{y,x} \nonumber\\
&=&  \left( \sum (h_{(1)}\cdot ( k\cdot {}_yf_x ) \right)_{y,x} \nonumber\\
&=&  \left( \sum (h_{(1)}\cdot {}_y1_y)( h_{(2)}k\cdot {}_yf_x ) \right)_{y,x} \nonumber\\
& = & \sum (h_{(1)} \bullet 1_{a(\mA)})(h_{(2)}k \bullet A)\nonumber
\end{eqnarray}
Therefore we can consider the partial smash product $\underline{a(\mA) \# H}$. The $k$-linear map 
$\Phi': a(\mA ) \ot H \rightarrow  a(\underline{\mA \# H})$ defined on generators by 
\beqnast
\Phi': a(\mA) \ot H & \rightarrow & a(\underline{\mA \# H})  \\
{}_yf_xE_{y,x} \ot h & \mapsto & 
({}_yf_x\# h)E_{y,x} 
\eqnast
induces the $k$-linear map 
\beqnast
\Phi: \underline{a(\mA) \# H} & \rightarrow & a(\underline{\mA \# H})  \\
{}_yf_xE_{y,x} \# h & \mapsto & 
({}_yf_x\# h)E_{y,x} 
\eqnast
which is an isomorphism of $k$-algebras. In fact, it is clear that its inverse is
\beqnast
\Psi:a(\underline{\mA \# H})  & \rightarrow &  \underline{a(\mA) \# H} \\
({}_yf_x\# h)E_{y,x}  & \mapsto & 
{}_yf_xE_{y,x} \# h
\eqnast
so it is enough to check that $\Phi$ is an algebra morphism. Given
\ ${}_wg_zE_{w,z} \# k$ and \ ${}_yf_xE_{y,x} \# h$ in $\underline{a(\mA) \# H}$, 
\beqnast 
\Phi({}_wg_zE_{w,z} \# k)\Phi({}_yf_xE_{y,x} \# h) 
& = & ({}_wg_z \# k)E_{w,z} \ \ ({}_yf_x\# h)E_{y,x} \\
& = & \delta_{z,y}({}_wg_z \# k)({}_yf_x\# h)E_{w,x} \\
& = & \sum \delta_{z,y}({}_wg_z (k_{(1)} \cdot{}_yf_x)\# k_{(2)}h) \ E_{w,x} . 
\eqnast 
On the other hand, 
\beqnast
({}_wg_zE_{w,z} \# k)({}_yf_xE_{y,x} \# h)
& = & \sum ({}_wg_zE_{w,z}) (k_{(1)} \bullet ({}_yf_xE_{y,x})) \# k_{(2)}h\\
& = & \sum ({}_wg_zE_{w,z}) ((k_{(1)} \cdot {}_yf_x)E_{y,x}) \# k_{(2)}h\\
& = & \sum \delta_{z,y} \ ({}_wg_z (k_{(1)} \cdot {}_yf_x))E_{w,x} \# k_{(2)}h
\eqnast
and it follows that 
\beqnst
\Phi(({}_wg_zE_{w,z} \# k)({}_yf_xE_{y,x} \# h)) =
\Phi({}_wg_zE_{w,z} \# k)\Phi({}_yf_xE_{y,x} \# h) . 
\eqnst
\end{proof}

\section{Morita Equivalence}\label{Morita}

Following \cite{CS}, we say that two categories $\mA$ and $\mB$ are Morita equivalent $(\mA \sim_{M} \mB)$ if and only if the corresponding module categories are equivalent. 

We start this section by recalling that given a $k$-category $\mA$, a left $\mA$-module $\cal M$ is a 
functor ${\cal M}:\mA \rightarrow \vect$, where 
\beqnast
{\cal M}:\mA_0 & \rightarrow &  (\vect)_0 \\
 x & \mapsto &  \ _{x}{\cal M}
\eqnast
 and, given $x,y \in \mA_0$,  
\beqnast
\cm: \ _{y}A_{x}  & \rightarrow & \Hom_k(\ _x \cm, \ _{y}\cm) \\
 \yfx  & \mapsto & \cm( \back \yfx):  \ _{x}m \mapsto \yfx \ _xm.
\eqnast

In other words, a left $\mA$-module $\cal M$ is a
collection of $k$-modules $\{ _{x}{\cal M} \}_{ x \in \mA_0}$ provided with a left action on the $k$-modules of morphisms of $\mA$, given by $k$-module maps $_{y}A_{x} \otimes_{k} \ _{x}{\cal M} \rightarrow _{y}{\cal M}$, where the image of $\yfx \ot \ _x m$ will be denoted by $_{y}f_{x} \ _{x}m$ and satisfying the usual axioms 
\beqnast
_{z}f_{y} (\ygx \ {}_xm ) & = & (_{z}f_{y} \ygx) {}_{x}m, \\
_{x}1_{x} \ _{x}m & =& \ _{x}m.
\eqnast

Let $\mC$ be a $k$-category, and for each  $x \in \mC_0$ let $\back \xLx$ be an algebra which is Morita equivalent to $\back \xCx$. Let us fix a Morita context for each $x$: let $Q^x$ be a $(\back \xCx,\xLx)$-bimodule and $P^{x}$ be a $(\back \xLx,\xCx)$-bimodule; let
$\tau_x: Q^{x} \otimes_{\xLx} P^{x} \rightarrow \xCx$ be an isomorphism of $\back \xCx$-bimodules and $\sigma_x:  P^{x} \otimes_{\xCx} Q^{x} \rightarrow \xLx$ be an isomorphism of $\back \xLx$-bimodules such that 
\[\sigma(p \otimes q) p' = p \tau(q \otimes p')\\ \hbox{ and} \\  
\tau(q \otimes p) q' = q \sigma(p \otimes q').\]

Then, we can define a new $\mC_0$-category $\mD = \morita$ in the following way: 
\[\yDx = P^{y} \underset{\yCy}{\otimes} \yCx \underset{\xCx}{\otimes} Q^{x},\]
with composition defined by  
\[
(p^{z} \otimes \zfy \otimes q^{y}) \circ (p'^{y} \otimes \ygx \otimes q'^{x}) = p^{z} \ot (\zfy  \circ \tau_y (q^{y} \otimes p'^{y}) \circ \ygx) \otimes q'^{x},
\]
which is easily seen to be associative.
Note that $\back \xDx \simeq \back \xLx$ as $k$-algebras. By \cite[Prop. 4.4]{CS}, the $k$-categories $\mC$ and $\morita$ are Morita equivalent, where the equivalence is provided by the functor which fixes objects and acts as $P^y\otimes_{{}_y\mC_y} \hbox{--} \otimes_{{}_x\mC_x}Q^x$ on $\yCx$  for each pair of objects $x$, $y\in \mC_0$.

It is proved in \cite{AB1} that if $A$ is a partial $H$-module algebra and $B$ is a globalization, then the algebras $\underline{A \# H}$ and $B \# H$ are Morita equivalent. From this it follows that if $\mA$ is a partial $H$-module category and $\mB$ is a globalization then, for each $x \in \mA_0 = \mB_0$, the algebras $\back \xASHx$ and $_x(\cb \# H)_x$ are Morita equivalent. This is the first step in a construction of  
\cite{CS} of Morita equivalent categories, and here it will provide a Morita equivalence between $\ASH$ and $\mB \# H$.

\begin{thm}Let $\mA$ be a partial $H$-module category and let $\mB$ be a globalization; then $\underline{\mA \# H}$ and ${\mB} \# H$ are Morita equivalent.
\end{thm}
\begin{proof}
Let $\mA$ an partial $H$-module category and $({\cal B},F)$ a globalization of $\mA$. Identifying $\mA$ and $F(A)$ via $F$ as before, $\ASH$ can be considered as a subcategory of $\mB \# H$. As we have already mentioned before, for each $x\in \mA_0$ the pair $({}_x\mB_x ,F)$ is indeed a globalization for the partial action of $H$ on the algebra ${}_x\mA_x$, then, by theorem 4 in the reference \cite{AB1} we conclude that 
\[\xASHx \sim_{M} \ _{x}(\mB \# H)_{x}\]
for each $x \in \mA_0$. 

A Morita context between $\xASHx$ and $_x(B \# H)_x$ is given by the \linebreak $(\back \xASHx,\ _{x}(\mB \# H)_{x})$-bimodules $M^{x}$ and $(\ _{x}(\mB \# H)_{x}, \xASHx)$-bimodules $N^x $ defined by 
\[
M^{x} = (_{x}\mA_{x} \# 1_{H}) (_{x}1_{x} \# H) = \{ \sum \ _{x}^{}f_{x}^{i} \otimes h_{i} | \ _{x}^{}f_{x}^{i} \in \back \xAx, h_i \in H \}
\]
and
\[
N^{x}  = (_{x}1_{x} \# H)  (_{x}\mA_{x} \# 1_{H}) =  \{ \sum(h^{i}_{(1)} \rhd \ _{x}^{}f_{x}^{i}) \otimes h_{(2)}^{i} | \ _{x}^{}f^{i}_{x} \in _{x}\mA_{x}, h^i \in H \} ,
\]
both subspaces of $_x(\mB \# H)_x$. The bimodule structure in both of them is given by the composition in $ _x(\mB \# H)_x$. 

Now let us take a look at $\mD = \mD(M,\mB \# H, N)$, which we know to be Morita equivalent to $\mB \# H$. Its objects are the same objects of $\ASH$ and, with respect to the morphisms, since the first tensor in the definition of $_{y}\mD_{x}$ is defined over $_y\mB_y \# H$, the second over $_x\mB_x \# H$, and the bimodules $M^y$ and $N^x$ live inside  $_y\mB_y \# H$ and $_x\mB_x \# H$ respectively, we have 
\beqnast
_{y}\mD_{x} & = & M^{y} \otimes \ _{y}(\mB \# H)_{x} \otimes N^{x} \\
& = & (\yAy \# 1_H)(\xidx \# H) \otimes \ _{y}(\cal B \# H)_{x} \otimes (\xidx \# H)(\xAx \# 1_H) \\
& = & (\yidy^A \# 1_H) \ot (\yAy \# 1_H)(\xidx \# H)  \ _{y}(\cal B \# H)_{x}  (\xidx^A \# H)(\xAx \# 1_H) \ot(\xidx^A \# 1_H) \\
& = & (\yidy^A \# 1_H) \ot \ _y (\underline{\mA \# H})_x \ot (\xidx^A \# 1_H)\\
& \simeq & \yASHx
\eqnast
as vector spaces. Moreover, composition in $\mD$ is given by 
\beqnast
& & [(\zidz^A \otimes 1_H) \ot  ( \zfy \# h ) \ot (\yidy^A \otimes 1_H)] \circ  
[(\yidy^A \ot 1_H) \ot  ( \ygx \#  k ) \ot (\xidx^A \ot 1_H)] \\ 
& = & (\zidz^A \ot 1_H) \ot  \sum \zfy (h_{(1)} \cdot \ygx) \otimes h_{(2)}k \ot (\xidx^A \ot 1_H) 
\eqnast
and this shows that there is an obvious isomorphism of categories $\cal{G}: \ASH \rightarrow \mD$, which is the identity on $\ASH_0=\mD_0$ and is defined on morphisms by 
\beqnast
\cal{G}: \yASHx & \rightarrow & _{y}{\mD}_{x} \\
\yfx \# k  & \mapsto & (\yidy^A \ot 1_H) \ot ( \yfx \# k ) \ot (\xidx^A \ot 1_H)
\eqnast
Therefore,
\[
\ASH {\cong} \mD \underset{\cal{M}}{\sim} \mB \# H . 
\]
\end{proof}

\section*{Acknowledgements}

The authors would like to thank Andrea Solotar for fruitful discussions during her visits to UFPR, and Mariano Suarez-Alvarez for his suggestions, which shaped one of  the main ideas in this paper.

\footnotesize
\noindent E.R.A.:
\\ Centro Polit\'ecnico, Departamento de Matem\'atica,\\ 
Universidade Federal do Paran\'a,\\
CP 019081, Jardim das Am\'ericas, Curitiba-PR,
81531-990, Brasil.
{\tt rolo1rolo@gmail.com}

\noindent M.M.S.A.:
\\ Centro Polit\'ecnico, Departamento de Matem\'atica,\\ 
Universidade Federal do Paran\'a,\\
CP 019081, Jardim das Am\'ericas, Curitiba-PR,
81531-990, Brasil.
{\tt marcelomsa@ufpr.br}

\noindent  E.B.:
\\Departamento de Matem\'atica,\\
Universidade Federal de Santa Catarina,\\
Campus Trindade, Florian\'opolis-SC, 
88040-900, Brasil.
{\tt ebatista@mtm.ufsc.br}


\begin{thebibliography}{99}

\bibitem[A]{A} 
F. Abadie; {\it Enveloping Actions and Takai Duality for Partial Actions}, J. Funct. Analysis {\bf 197} (2003), no. 1, 14--67.

\bibitem[AB1]{AB1} 
M.M.S. Alves and E. Batista; {\it Enveloping Actions for Partial Hopf Actions}, Comm. Algebra {\bf 38} (2010), 2872--2902.

\bibitem[AB2]{AB2} 
M.M.S. Alves and E. Batista; {\it Globalization theorems for partial Hopf (co)actions, and some of their applications},  
Contemporary Mathematics  {\bf 537} (2011), 13--30.

\bibitem[BG]{BG}
K. Bongartz and P. Gabriel; {\it Covering spaces in representation theory}, Invent. Math. {\bf 65} (1981/82), 331--378, . 

\bibitem[CJ]{CJ} 
S. Caenepeel and K. Janssen; {\it Partial (co)actions
of Hopf algebras and partial Hopf-Galois theory}, Comm. Algebra {\bf 36}
(2008), 2923--2946.

\bibitem[CM]{CM}
C. Cibils and E. Marcos; {\it Skew category, Galois
covering and smash product of a category over a ring}, Proc. Amer.
Math. Soc.  {\bf 134} (2006),  no. 1, 39--50.

\bibitem[CoM]{CoM}
M. Cohen; S. Montgomery; 
{\it Group-graded rings, smash products, and group actions}, 
 Trans. Amer. Math. Soc. 300 (1987), 810-811. 

\bibitem[CFM]{CFM}
W. Cortez, M. Ferrero and E. Marcos; {\it Partial actions on categories}, arXiv:1107.3850 (2011).

\bibitem[CRS]{CRS}
C. Cibils, M.J. Redondo and A. Solotar; {\it  Connected gradings and fundamental group},  
Algebra Number Theory {\bf 4} (2010), no. 5, 625--648.

\bibitem[CS]{CS}
C. Cibils and A. Solotar; {\it Galois coverings, Morita equivalence and smash extensions of categories over a field}, 
Doc. Math. {\bf 11} (2006), 143--159.

\bibitem[DE]{DE}
M. Dokuchaev and R. Exel; {\it Associativity of crossed products by partial actions, enveloping actions and 
partial representations}, Trans. Amer. Math. Soc. {\bf 357} (2005), no. 5, 1931--1952.

\bibitem[E]{E}
R. Exel; {\it Circle actions on $C^*$-algebras, partial automorphisms, and a generalized Pimsner-Voiculescu exact sequence}, 
J. Funct. Anal. {\bf 122} (1994), no. 2, 361--401.

\bibitem[E2]{E2}
R. Exel; {\it Partial actions of groups and actions of inverse semigroups}, 
Proc. Amer. Math. Soc. {\bf 126} (1998), no. 12, 3481--3494.

\bibitem[FL]{FL}
M. Ferrero and J. Lazzarin; {\it Partial actions and partial skew group rings}, 
J. Algebra {\bf 319} (2008), 5247--5264. 

\bibitem[G]{G} P. Gabriel; { \it Auslander-Reiten sequences and Representation-finite Algebras},  In Proc. ICRA II.
(Ottawa 1979). Representations of Algebras, number 831 in Lecture Notes in Mathematics,
pages 1–71, Ottawa-Canada, Springer-Verlag, (1980). 


\bibitem[GS]{GS}
J.L. Garc\'{\i}a and J.J. Sim\'on; {\it Morita equivalence for idempotent rings}, J. Pure and Appl. Algebra {\bf 76} (1991), 
no. 1, 39--56.


\bibitem[Lo]{Lo}
C. Lomp; {\it Duality for partial group actions}, Int. El. J. Algebra {\bf 4} (2008), 53--62. 


\bibitem[MVP]{MVP} R. Martínez-Villa and J.A. de la Peña; {\it The universal cover of a quiver with relations}, J. Pure
Appl. Algebra, 30:873–887, 1983.


\bibitem[M]{M}
B. Mitchell; {\it The dominion of Isbell}, 
 Trans. Amer. Math. Soc.  {\bf 167}  (1972), 319--331. 


\end{thebibliography}
\end{document}